%% file: Paper_arXiv.tex
\documentclass[11pt]{article}

\input{head/package}

\input{head/thm}
\input{head/macro}
\usepackage{amsmath}
\usepackage{comment}

\usepackage{CJKutf8}

\usepackage{datetime}
\date{\vspace{-2.5\baselineskip}}

\author[1]{Yuya Hikima\footnote{Corresponding author. E-mail: \url{yuya-hikima@g.ecc.u-tokyo.ac.jp}}}
\author[1,2]{Akiko Takeda}

\affil[1]{Graduate School of Information Science and Technology, University of Tokyo, Tokyo, Japan}
\affil[2]{Center for Advanced Intelligence Project, RIKEN, Tokyo, Japan}

\title{Zeroth-Order Methods for Nonconvex Stochastic Problems with Decision-Dependent Distributions}

\begin{document}
\maketitle

\begin{abstract}
In this study, we consider an optimization problem with uncertainty dependent on decision variables, which has recently attracted attention due to its importance in machine learning and pricing applications.
In this problem, the gradient of the objective function cannot be obtained explicitly because the decision-dependent distribution is unknown. Therefore, several zeroth-order methods have been proposed, which obtain noisy objective values by sampling and update the iterates. 
Although these existing methods have theoretical convergence for optimization problems with decision-dependent uncertainty, they require strong assumptions about the function and distribution or 
exhibit large variances in their gradient estimators.
To overcome these issues, we propose two zeroth-order methods under mild assumptions.
First, we develop a zeroth-order method with
a new one-point gradient estimator including a variance reduction parameter.
The proposed method updates the decision variables while adjusting the variance reduction parameter. 
Second, we develop a zeroth-order method with a two-point gradient estimator. 
There are situations where only one-point estimators can be used, but if both one-point and two-point estimators are available, it is more practical to use the two-point estimator.
As theoretical results, we show the convergence of our methods to stationary points and provide the worst-case iteration and sample complexity analysis.
Our simulation experiments with real data on a retail service application show that our methods output solutions with lower objective values than the conventional zeroth-order methods.
\end{abstract}

\section{Introduction}
In this study, we consider the following problem:
\begin{align*}
{\rm (P)} \quad \min_{\bx \in \R^d} \quad F(\bx):=\E_{\bmxi \sim D(\bx)}[f(\bx,\bmxi)],  
\end{align*}
where $F: \R^d \to \R$ is generally non-convex.
The main feature of this problem is that the probability distribution $D(\bx)$ depends on the decision variable $\bx$.
This feature appears in a wide range of applications such as performative prediction  \citep{perdomo2020performative} and 
price optimization \citep{ray2022decision,hikima2023stochastic}.
Because of its importance for applications, various studies have addressed this problem \citep{perdomo2020performative, mendler2020stochastic, ray2022decision, chen2023performative}.\footnote{Codes of our experiments can be found in https://github.com/Yuya-Hikima/AAAI25-Zeroth-Order-Methods-for-Nonconvex-Stochastic-Problems-with-Decision-Dependent-Distributions.}

In general, problem (P) is a non-convex problem:
even if $f(\bx,\bmxi)$ is convex with respect to $\bx$, it is generally non-convex since the probability distribution $D(\bx)$ depends on $\bx$.
Moreover, there exists a difficulty:
we cannot access the gradient of the objective function because the probability distribution $D(\bx)$ is unknown in practical applications.

Recently, several studies \citep{ray2022decision,liu2023time,chen2023performative} have proposed zeroth-order methods for (P).
Concretely, \cite{ray2022decision} proposed a condition under which the objective function of (P) is strongly convex, and employed a zeroth-order method based on the following one-point gradient estimator:
\begin{align*}
    &\bg_k:= \frac{d}{\mu} f(\bx_k+\mu \bv_k,\bmxi_k) \bv_k,%\label{eq:one-point_gradient_estimator}
\end{align*}
where $\bv_k$ is a unit vector sampled uniformly from the unit sphere in dimension $d$, $\bmxi_k \sim D(\bx_k+\mu \bv_k)$, and $\mu \in \R_{\ge 0}$. 
\cite{liu2023time} considered non-convex problem (P) with a state-dependent setting where the distribution evolves according to an underlying controlled Markov chain. 
They proposed a zeroth-order method based on the above one-point gradient estimator and showed its convergence to a stationary point.
\cite{chen2023performative} supposed that their optimization problem can be reduced to convex ones, and they employed a zeroth-order method based on the following gradient estimator:
\begin{align*}
    \bg_k:= &\frac{d}{2\mu} \left(f(\bx_k+\mu \bv_k,\bmxi^1_k)-f(\bx_k - \mu \bv_k,\bmxi^2_k)\right) \bv_k, 
\end{align*}
where $\bv_k$ is a unit vector sampled uniformly from the unit sphere in dimension $d$, $\bmxi^1_k \sim D(\bx_k+\mu \bv_k)$, $\bmxi^2_k \sim D(\bx_k-\mu \bv_k)$, and $\mu \in \R_{\ge 0}$.

Although these existing studies have successfully proposed zeroth-order methods with theoretical convergence for problem (P), there exist limitations.
First, zeroth-order methods with one-point gradient estimators \citep{ray2022decision,liu2023time} exhibit large variances of their gradient estimators, which may increase the number of iterations and sampling.
Specifically, although the variances of their gradient estimators are bounded by using $G:=\sup_{\bx,\bmxi}|f(\bx,\bmxi)|$, $G$ is generally large or unbounded.
Second, the existing zeroth-order method \citep{chen2023performative} with a two-point gradient estimator supposes that the distribution $D(\bx)$ is included in a special exponential family and the function $f$ is required to satisfy a certain (intuitively difficult to understand) inequality \citep[Section~3, Appendix F]{chen2023performative}.

In this paper, we propose two zeroth-order methods under mild assumptions, which do not restrict $f$ and $D(\bx)$ to any particular one.
First, we develop a new zeroth-order method with
an improved one-point gradient estimator.
Specifically, we propose the following gradient estimator including variance reduction parameter $c_k$:
\begin{align}
    \bg_k:= \frac{1}{\mu_k} \left(f(\bx_k+\mu_k \bu_k,\bmxi_k)-c_k\right) \bu_k, \label{eq:our_one-point_gradient_estimator}
\end{align}
where $\bu_k$ is sampled from the standard normal distribution $\mathcal N(0, I_d)$ in dimension $d$, $\bmxi_k \sim D(\bx_k+\mu_k \bv_k)$, $\mu_k \in \R_{>0}$, and $c_k \in \R$.
Then, we theoretically show that if $c_k$ is close to $\E_{\bmxi \sim D(\bx_k)}[f(\bx_k,\bmxi)]$, the variance of the above gradient estimator is reduced.
Based on this fact, we propose a zeroth-order method that updates the iterates $\bx_k$ using the proposed gradient estimator \eqref{eq:our_one-point_gradient_estimator} while adjusting $c_k$.
Second, we develop a zeroth-order method with
the following two-point gradient estimator:
\small
\begin{align}
    \bg_k:= &\frac{1}{2\mu_k} \left(f(\bx_k+\mu_k \bu_k,\bmxi^1_k)-f(\bx_k - \mu_k \bu_k,\bmxi^2_k)\right) \bu_k, \label{eq:our_two-point_gradient_estimator}
\end{align}
\normalsize
where $\bu_k$ is sampled from the standard normal distribution $\mathcal N(0, I_d)$ in dimension $d$, $\bmxi^1_k \sim D(\bx_k+\mu \bv_k)$, $\bmxi^2_k \sim D(\bx_k-\mu \bv_k)$, and $\mu_k \in \R_{>0}$. 
Although our method is similar to the zeroth-order method employed by \citep{chen2023performative}, it does not assume strong assumptions for the distribution $D(\bx)$ and the function $f$.
Moreover, we also include the Gaussian homotopy technique \citep{iwakiri2022single} in both proposed methods to find potentially better stationary points.

As theoretical results, we show the convergence of our methods to stationary points and provide the worst-case iteration and sample complexity analysis.
In particular, we show that the worst-case sample complexity of our methods is $O(d^{\frac{9}{2}}\epsilon^{-6})$.
Since the sample complexity of the existing zeroth-order method \citep{liu2023time} for the non-convex problem (P) is $O(G^6 d^2 \epsilon^{-6})$, the proposed method has an advantage when $G=\sup_{\bx,\bmxi}|f(\bx,\bmxi)|$ is large or  unbounded.

We conducted simulation experiments with real-data on a retail service application. 
The results show that our methods output solutions with lower objective values than
conventional zeroth-order methods.

Our contributions are as follows.
\begin{itemize}
\item We propose a new zeroth-order method with an improved one-point gradient estimator. 
When $G:=\sup_{\bx,\bmxi}|f(\bx,\bmxi)|$ is large or unbounded,  the sample complexity of the proposed method  has an advantage over existing zeroth-order methods \citep{ray2022decision,liu2023time}, which also utilize one-point gradient estimators.
\item We develop a zeroth-order method with a
two-point gradient estimator under milder assumptions than the existing ones \citep{chen2023performative}, which utilize a two-point gradient estimator.
Although (P) is generally non-convex under our loose assumptions, we show the convergence of the method to stationary points and provide its sample complexity.
\end{itemize}

\paragraph{Notation.}
Bold lowercase symbols (e.g., $\bx, \by$) denote vectors, and $\|\bx\|$ denotes the Euclidean norm of a vector $\bx$.
%The inner product of the vectors $\bx, \by$ is denoted by $\bx^\top \by$.
Let $\mR_{>0}$ ($\mR_{\ge 0}$) be the set of positive (non-negative) real numbers.
The gradient for a real-valued function $F(\bx)$ is denoted by $\nabla F(\bx)$.
A binomial coefficient of a pair of integers $m$ and $n$ is written as $\bigl(
\begin{smallmatrix}
   m \\
   n
\end{smallmatrix}
\bigl)$.
Let $[N]$ be the set of $\{1,2,\dots,N\}$.
The standard normal distribution in dimension $d$ is written as $\mathcal N(0, I_d)$.

\section{Related Work}
\subsection{Zeroth-order Methods} 
Zeroth-order methods are a class of powerful optimization tools to solve many complex  problems, whose explicit gradients are
infeasible to access.
Various types of zeroth-order methods have been proposed and their convergence has been analyzed \citep{ghadimi2013stochastic, nesterov2017random}.
Since then, multiple techniques for zeroth-order methods have been proposed, such as techniques to reduce the variance of the gradient estimator \citep{liu2018zeroth,ji2019improved}, to make the methods scalable through variable selection \citep{wang2018stochastic}, and to deal with regularized problems \citep{cai2022zeroth}.

Recently, zeroth-order methods have also been proposed for problems with decision-dependent uncertainty \citep{ray2022decision,liu2023time,chen2023performative}.
\cite{ray2022decision} and \cite{liu2023time} proposed zeroth-order methods using one-point gradient estimators for (P) where the distribution changes dynamically with time.
\cite{chen2023performative} showed that (P) can be reduced to a convex optimization problem for particular combinations of the distribution $D(\bx)$ and the function $f$.
They then employed a zeroth-order method using a two-point gradient estimator.
In this study, we propose a new zeroth-order method that extends the conventional one-point gradient estimator used in \citep{ray2022decision,liu2023time} by including a variance reduction parameter.
Moreover, we also develop a zeroth-order method with a two-point gradient estimator without strong assumptions on the distribution and function.

\subsection{Other Methods for Stochastic Problems with Decision-dependent Uncertainty} \label{subsec:related_decision_dependent_noise}
We describe different techniques, other than zeroth-order methods, for solving (P).\footnote{
Another formulation dealing with decision-dependent uncertainties is \emph{decision-dependent distributionally robust optimization} \citep{luo2020distributionally,basciftci2021distributionally}, which aims to find an optimal solution in the worst case when the probability distribution has ambiguity.
Although such formulations are effective when the ambiguity set of the probability distribution is known,
we consider situations where the probability distribution is completely unknown.}

\paragraph{Retraining methods \citep{perdomo2020performative,mendler2020stochastic}.}
Retraining methods update the current iterate by fixing the distribution at each iteration.
Specifically, \cite{perdomo2020performative} proposed \textit{repeated gradient descent}: $\bx_{k+1}:= \mathrm{proj}_{\mathcal C}(\bx_k - \eta_k \E_{\bmxi \sim D(\bx_k)} [\nabla_{\bx} f(\bx_k,\bmxi)]),$
where $\mathcal C$ is the feasible region and $\mathrm{proj}_{\mathcal C}$ is the Euclidean projection operator onto $\mathcal C$.
They showed that the repeated gradient descent method converges to a \emph{performatively stable point} defined as $\bx_{\mathrm{PS}} := \mathrm{arg} \min_{\bx} 
\E_{\bmxi \sim D(\bx_{\mathrm{PS}})}[f(\bx,\bmxi)]$.
However, these methods assume the strong convexity of $f(\bx,\bmxi)$ \emph{w.r.t.} $\bx$ and are not applicable to our problem.

\paragraph{Stochastic gradient descent methods \citep{hikima2023stochastic,sutton2018reinforcement}.}
These methods update the current iterate by the stochastic gradient such as $\bx_{k+1}:= \bx_k - \eta_k \bm{g}^\dagger_k$, where $\bm{g}^\dagger_k$ is an unbiased stochastic gradient for the objective function.
Although these methods converge to stationary points, they require access to $\mpr(\bmxi \mid \bx)$ and $\nabla_{\bx} \mpr(\bmxi \mid \bx)$ to find an unbiased stochastic gradient $\bm{g}^\dagger_k$.
Since we assume that the probability density function of $D(\bx)$ is unknown, these methods are not applicable to our problem. 

\paragraph{Two-stage approach \citep{miller2021outside}.}
This approach estimates a model of the distribution map $D(\cdot)$ in the first stage and optimizes the proxy function of the objective function by using the estimated distribution in the second stage.
Since this approach assumes that the distribution map is included in location-scale families \citep[Eq. (2)]{miller2021outside}, it is not applicable to our problem. 

\paragraph{Bayesian optimization \citep{brochu2010tutorial,frazier2018tutorial}.}
This method is a process of learning for global optimization of black-box functions.
When this method is applied to our problem, it is necessary to obtain a large number of samples $\bmxi \sim D(\bx)$ in order to widely search the entire space of decision variables.
This is not desirable from a practical point of view, since it is necessary to deploy decisions ($\bx$) in the real-world to obtain $\bmxi \sim D(\bx)$.

\section{Preliminaries}
\subsection{Problem Definition} \label{subsec:problem_def}
We restate the problem under consideration:
\begin{align*}
\textrm{(P)} \quad \min_{\bx \in \R^d} \ F(\bx):=\E_{\bmxi \sim D(\bx)}[f(\bx,\bmxi)],
\end{align*}
where $F$ is generally non-convex.
Note that problem (P) is a broad problem class that also includes the following optimization problem:
\begin{align*}
\min_{\bx \in \R^d} \ \E_{\bmxi \sim D}[f(\bx,\bmxi)].
\end{align*}

Our goal is to obtain better solutions with low sample complexity for $D(\bx)$.
This is because, to obtain a sample $\bmxi \sim D(\bx)$, a decision ($\bx$) must be deployed in the real-world (e.g., a decision maker must sell some products at price $\bx$ to obtain a sample of demand $\bmxi\sim D(\bx)$), so the fewer samples the better.

\subsection{Assumptions}
We make the following assumptions.
\begin{assumption} \label{asm:F_sample_var_bound}
For any $\bx \in \R^d$, there exists a constant $\sigma \in \R_{\ge 0}$ satisfying 
\begin{align}
\E_{ \bmxi \sim D(\bx)} [ \left(F(\bx)-f(\bx, \bmxi)\right)^2] \le \sigma^2. \label{eq:obj_val_variance}    
\end{align}
\end{assumption}

\begin{assumption} \label{asp:distribution_dist_bound}
For any $\bx \in \R^d$, there exists a constant $\alpha \in \R_{\ge 0}$ satisfying
$$W(D(\bx),D(\bx')) \le \alpha \| \bx -\bx'\|, $$
where $W$ represents the Wasserstein-1 distance.
\end{assumption}

\begin{assumption}\label{asp:f_Lipshitz}
%For any $\bx \in \R^d$, 
$f(\bx,\bmxi)$ is $L_{\bmxi}$-Lipschitz continuous in $\bmxi$.
Moreover, 
%for any $\bmxi$, 
$f(\bx,\bmxi)$ is $L_{\bx}$-Lipschitz continuous in $\bx$.
\end{assumption}

\begin{assumption} \label{asm:F_smooth}
$F(\bx)$ is $H_F$-smooth.
\end{assumption}

For Assumption \ref{asm:F_smooth},
the following lemma has been given as a sufficient condition.
\begin{lemma} \citep[Lemma 1]{ray2022decision}
Suppose that there exist matrix $\bA$ and distribution $D'$ such that
$$ \bmxi \sim D(\bx) \Longleftrightarrow \bmxi = \bnu + \bA\bx,$$
where $\bnu$ has mean $\bar{\bnu} :=\E_{\bnu \sim D'}[\bnu]$ and co-variance $\E_{\bnu \sim D'}[(\bnu-\bar{\bnu})(\bnu-\bar{\bnu})^{\top}]$.
Moreover, suppose that $f(\bx,\bmxi)$ is $\rho$-smooth with respect to both $\bx$ and $\bmxi$.
Then, Assumption~\ref{asm:F_smooth} holds with
$$H_F:=\sqrt{\rho^2 (1+\|A\|^2_{\textrm{op}}) \max (1, \|A\|^4_{\textrm{op}})},$$
where $\|A\|_{\textrm{op}}$ is the operator norm of $A$.
\end{lemma}

Our assumptions are looser than those of an existing study \citep{ray2022decision} that tackles a problem similar to ours.
Assumption 5 in \citep{ray2022decision} implies that 
$E_{ \bmxi \sim D(\bx)} [ \left(F(\bx)-f(\bx, \bmxi)\right)^2] \le (2G)^2$ for $G:=\sup_{\bx,\bmxi}|f(\bx,\bmxi)|$, which yields our Assumption \ref{asm:F_sample_var_bound}.
Moreover, Assumptions 1 and 3 in \citep{ray2022decision} yields our Assumptions~\ref{asp:distribution_dist_bound}--\ref{asm:F_smooth}.
Conversely, we do not require Assumption 1(c) and Assumption 2 in \citep{ray2022decision}.
The price of our loose assumptions is that our problem becomes a non-convex problem.
Therefore, we develop zeroth-order methods that converge to stationary points, rather than optimal solutions.

\subsection{Gaussian Smoothed Function}
To propose our method, let us define the Gaussian smoothed function for $F$.
\begin{definition}
We call the following function the Gaussian smoothed function of $F$.
$$F_{\mu}(\bx):=\E_{\bu\sim \mathcal N(0, {\mathrm I}_d)}[F(\bx+\mu \bu)].$$    
\end{definition}

\noindent Function $F_\mu$ is an approximation of $F$.
Here, the smoothing parameter $\mu$ controls the level of approximation: when $\mu$ is small, $F_{\mu}$ is close to $F$.

In this paper, we propose two unbiased gradient estimators for $F_{\mu}$.
Our methods update $\bx_k$ using the proposed gradient estimators while reducing the smoothing parameter $\mu$.

\section{Proposed Method with One-point Gradient Estimator}
\subsection{One-point Gradient Estimator} 
\label{sec:unbiased_gradient}
We propose the following one-point gradient estimator:\footnote{The term ``one-point gradient estimator'' in this paper refers to the gradient estimator obtained by queries of the objective function at a single point ($\bx+\mu \bu$). 
Note that ``one-point gradient estimator'' is also used to refer to a gradient estimator obtained by one query of the objective function, but in our case, Algorithm 1 proposed later requires additional calculations of $f(\bx,\bmxi)$ to set parameter $c$. Moreover, if we use mini-batch gradients, Algorithm 1 requires multiple computations of $f(\bx,\bmxi)$ for different $\xi$.}
\begin{align}
\bg_1(\bx,\mu, c, \bu,\bmxi):= \frac{f(\bx+\mu \bu, \bmxi)-c}{\mu}\bu, \label{eq:one-point_grad}    
\end{align}
where $\bu \sim \mathcal N(0, {\mathrm I}_d)$ and $\bmxi \sim D(\bx+\mu \bu)$.
Here, $\bg_1(\bx,\mu, c, \bu,\bmxi)$ is an unbiased gradient estimator of the function $F_{\mu}$ as stated in the following lemma.
\begin{lemma} \label{lem:unbiased_gradient_gaussian_smoothing}
For any $\bx \in \R^d, \mu \in \R_{>0}$, and $c\in \R$, the following holds.
$$ \E_{u\sim \mathcal N(0, {\mathrm I}_d)} \left[ \E_{\bmxi \sim D(\bx+\mu \bu)}\left[\bg_1(\bx,\mu,c,\bu,\bmxi) \right]\right] = \nabla F_{\mu}(\bx). $$
\end{lemma}

The above gradient estimator is a generalization of the one-point gradient estimator in \citep{liu2023time}:
when $c=0$, our estimator is the same as that of \citep[Eq. (4)]{liu2023time}.
While $c$ does not affect the unbiasedness of the gradient estimator, it affects its variance.
This fact is shown in the following lemma.
\begin{lemma} \label{lem:g_bound}
Suppose that Assumptions~\ref{asm:F_sample_var_bound} and \ref{asm:F_smooth} hold.
Then, the following holds for any $\bx \in \R^d$, $c \in \R$, and $m \in \mathbb{N}$.
\begin{align*}
&\E_{\bu \sim \mathcal N(0, {\mathrm I}_d)}\left[  \E_{ \{\bmxi^j\}_{j=1}^{m} \sim D(\bx+\mu \bu)} \left[ \left\|\frac{1}{m}\sum_{j=1}^{m} \bg_1(\bx,\mu,c,\bu,\bmxi^j) \right\|^2  \right]\right] \\
&\le  \frac{3}{2}\mu^2H_F^2 (d+6)^3 +6(d+4)\|\nabla F(\bx)\|^2 + \frac{3\sigma^2d}{\mu^2 m} + \frac{3d(F(\bx)-c)^2}{\mu^2}.
\end{align*}
\end{lemma}
From Lemma \ref{lem:g_bound}, when $c$ is close to $F(\bx)$, the variance of our gradient estimator is reduced.
It is widely known in the field of stochastic optimization that reducing the variance of the gradient estimator has a positive impact on the convergence of the method \citep{johnson2013accelerating}.
Therefore, in the next section, we give a method for setting $c_k$ at each iteration $k$ to reduce the variance of the proposed gradient estimator.

\paragraph{Discussion.}
In the existing studies \citep{ray2022decision,liu2023time}, the variance of their gradient estimators (\citep[(3)]{ray2022decision} and \citep[(4)]{liu2023time}) are bounded by using $G:= \max_{\bx, \bmxi} f(\bx,\bmxi)$.
However, since $G$ is generally large, their  upper bounds are also generally large.
In contrast to existing studies, %\eqref{eq:exsiting_one-point_g_upperbound},
Lemma~\ref{lem:g_bound} provides an upper bound independent of $G$.

\subsection{Setting of Variance Reduction Parameter $c$} \label{subsec:c_setting}
To reduce the variance in Lemma \ref{lem:g_bound}, the parameter $c$ should be close to $F(\bx)$.
Therefore, during the iterations of the algorithm, we update $c_k$ to bring it closer to the target value $F(\bx_k)$. 
A naive method is to obtain samples $\{ \bmxi^j_k\}_{j=1}^{j_{\max}} \sim D(\bx_k)$ for some $j_{\max} \in \mathbb{N}$, and set $c_k$ as follows:
$$ c_k:= \frac{1}{j_{\max}} \sum\nolimits_{j=1}^{j_{\max}} f(\bx_k, \bmxi^j_k).$$
However, this method requires new samples to estimate $c_k$, increasing the sample complexity.
We propose an approach, inspired by \citep{jagadeesan2022regret}, to approximate $F(\bx_k)$ from samples obtained in past iterations.
Specifically, using the sample set $\{ \bmxi^j_i \}_{j=1}^{m_i} \sim D(\bx_i + \mu_i \bu_i)$ at each past $i \in \{k-s, \dots, k-1\}$ iteration, which is obtained to compute the gradient estimator, we compute $c_k$ as follows.
\begin{align}
    c_{k} := \sum\nolimits_{i=k-s}^{k-1} \frac{a_i}{m_i}\sum\nolimits_{j=1}^{m_i}f(\bx_{k}, \bmxi_i^j), \label{eq:c_k}
\end{align}
where $a_i \in [0,1]$ and $\sum_{i=k-s}^{k-1} a_i =1$.
Here, $a_i$ represents the importance of the samples of the $i$-th iteration.
For example, if $\bx_{k}$ and $\bx_i+\mu_i \bu_i$ are close, then $D(\bx_{k})$ and $D(\bx_i+\mu_i \bu_i)$ have similar distributions from Assumption~\ref{asp:distribution_dist_bound}.
Then, the samples of the $i$-th iteration are important, and it is better to make $a_i$ larger.

To set $a_i$ for each $i \in \{k-s,\dots,k-1\}$ appropriately, we first show the following lemma.
\begin{lemma} \label{lem:delta_bound_2}
Suppose that Assumptions \ref{asm:F_sample_var_bound}--\ref{asp:f_Lipshitz} hold.
Let $c_k$ be defined by \eqref{eq:c_k}.
Then, the following holds:
\begin{align}
\left( F(\bx_k) - c_k \right)^2 
\le 2 \sum_{i=k-s}^{k-1} a_i^2 \left(L_{\bmxi}^2 \alpha^2 \| \bx_k -\bx_i - \mu_i \bu_i \|^2+  \frac{\sigma^2}{m_i}\right).\label{eq:F_x_k_c_k_bound}
\end{align}
\end{lemma}

From Lemma \ref{lem:delta_bound_2}, $a_i$ that minimizes the right-hand side of \eqref{eq:F_x_k_c_k_bound} is desirable to bring $c_k$ close to $F(\bx_k)$.
Such $a_i$ can be obtained in closed form from the following lemma.
\begin{lemma} \label{lem:weight_opt_2}
Consider the following optimization problem:
\begin{align*}
\min_{\ba} \ \ &\sum\nolimits_{i=k-s}^{k-1} a_i^2 \left( L_{\bmxi}^2 \alpha^2 \| \bx_k -\bx_i - \mu_i \bu_i \|^2+  \frac{\sigma^2}{m_i} \right)  \\
\mathrm{s.t.} \  \ &a_i \ge0,\  \forall i \in \{k-s, \dots, k-1\}, \\
& \sum\nolimits_{i=k-s}^{k-1} a_i=1.
\end{align*}
Then, the optimal solution is $\ba^*$ such that $a^*_i:= \frac{1}{b_i \sum_{j=k-s}^{k-1} \frac{1}{b_j}}$, where $b_i:= L_{\bmxi}^2 \alpha^2 \| \bx_k -\bx_i - \mu_i \bu_i \|^2+  \frac{\sigma^2}{m_i}$.
\end{lemma}
Using the results of Lemmas \ref{lem:delta_bound_2} and \ref{lem:weight_opt_2}, we can show the theoretical convergence of our method proposed in the next section.

\subsection{Proposed Method and Theoretical Results}
We propose Algorithm \ref{alg:zeroth_order_method}.
Lines \ref{line:sampling}--\ref{line:x_k_update} of  Algorithm \ref{alg:zeroth_order_method} update the iterate based on the gradient estimator \eqref{eq:one-point_grad}.
Line~\ref{line:mu_k_update} adjusts the smoothing parameter $\mu_k$:
Algorithm \ref{alg:zeroth_order_method} starts with a sufficiently large $\mu_0$ and gradually reduces $\mu_k$ so that the smoothed function $F_{\mu_k}$ approaches the original objective function $F$.
Lines \ref{line:s_setting}--\ref{line:c_setting} calculate the variance reduction parameter $c_{k+1}$ from the past samples.
We define $s_{\max}$ as the maximum window size, which indicates how far back in the past the samples are considered for calculating $c_{k+1}$.
Note that lines \ref{line:s_setting}--\ref{line:c_setting}, although requiring oracle computation of the function $f$, do not increase the sample complexity since those do not need new samples.

\begin{algorithm}[tb]
  \caption{Zeroth-order method with the improved one-point gradient estimator}
  \label{alg:zeroth_order_method}
  \begin{algorithmic}[1]
    \Require $\bx_0$, $\mu_0$, $\mu_{\min}$, $c_0$, $s_{\max}$, $\beta$, $\gamma$, $\{m_k\}_{k=0}^T$, $M$, $T$. 
    \For{$k =0, 1,\dots, T$}
    \State Sample $\bu_k$ from $\mathcal N(0, {\mathrm I}_d)$ and $\{\bmxi_k^j\}_{j=1}^{m_k}$ from $D(\bx_k+\mu_k \bu_k)$ \label{line:sampling}
    %\State $\bg_k \gets \frac{ \frac{1}{m_i}\sum_{j=1}^{m_i} f(\bx_k+\mu_k \bu_k, \bmxi_k^j)-c_k}{\mu_k}\bu_k$
    \State $\bg_k \gets \frac{1}{m_k}\sum_{j=1}^{m_k} \bg_1(\bx_k,\mu_k, c_k, \bu_k,\bmxi_k^j)$ \label{line:g_k}
    \State $\bx_{k+1}\gets \bx_k - \beta \bg_k$ \label{line:x_k_update}
    \State $\mu_{k+1}\gets \max(\gamma \mu_k, \mu_{\min})$ \label{line:mu_k_update}
    \State $s \gets \min(s_{\max},k+1)$ \label{line:s_setting}
    \State $a_i\gets \frac{1}{b_i\sum_{j=k-s+1}^{k}\frac{1}{b_j}}$ for all $i \in [k-s+1,k]$,  where $b_i:=M\|\bx_{k+1}-\bx_i-\mu_i \bu_i\|^2+\frac{1}{m_i}$ \label{def:a_i}
    \State $c_{k+1} \gets \sum_{i=k-s+1}^{k} \frac{a_i}{m_i}\sum_{j=1}^{m_i}f(\bx_{k+1}, \bmxi_i^j)$ \label{line:c_setting}
    \EndFor
  \end{algorithmic}
\end{algorithm}

\paragraph{Remark 1.}
The proposed method incorporates the Gaussian Homotopy technique \citep{iwakiri2022single} to obtain better stationary points.
Specifically, the method starts from solving an almost convex smoothed function $F_{\mu_0}(\bx)$ with sufficiently large $\mu_0 \ge 0$ and gradually changes the optimization problem $F_{\mu_k}(\bx)$ to the original one $F(\bx)$.
It is known that a better local solution can be potentially obtained by the technique \citep{mobahi2015link,mobahi2015theoretical,hazan2016graduated,iwakiri2022single}.

%\subsection{Theoretical Results} \label{subsec:theoretical_results}
To demonstrate the convergence of Algorithm \ref{alg:zeroth_order_method}, we first present the following lemma.

\begin{lemma} \label{lem:delta_bound}
Suppose that Assumptions \ref{asm:F_sample_var_bound}--\ref{asm:F_smooth} hold.
Let $\{\bx_k\}$ and $\{c_k\}$ be the sequence generated by Algorithm~\ref{alg:zeroth_order_method} with $m_k \ge m_{\min}$ for $k \in \{0, 1,\dots, T\}$, $\beta \le \frac{\mu_{\min}}{2L_{\bmxi}\alpha \sqrt{6d}}$, and $M:=\frac{L_{\bmxi}^2\alpha^2}{\sigma^2}$.
Then, for any setting parameter
$\bx_0 \in \R^d$, $\mu_0 \in \R_{>0}$, $\mu_{\min} \le \mu_0$, $c_0 \in \R$, $s_{\max} \in \mathbb{N}$, and $\gamma \in (0,1)$,
the following holds:
\begin{align*}
\E_{\zeta_{[0,k-1]}} [\delta_k^2] \le &\left(\frac{1}{2}\right)^{k} \delta_0^2  +  \frac{\mu_{\min}^2\mu_0^2H_F^2(d+6)^3}{2d} + 2\mu_{\min}^2 \frac{d+4}{d} L_F^2 + 8L_{\bmxi}^2 \alpha^2 \mu_0^2 d + \frac{5}{m_{\min}}\sigma^2,
\end{align*}
where $\delta_k:=F(\bx_k)-c_k$, $k \in \{1,\dots, T\}$, and $\zeta_{[0,k-1]}:=\{\bu_{0}, \bmxi_{0}, \dots, \bu_{k-1}, \bmxi_{k-1} \}$.
\end{lemma}
The above lemma shows that the estimation error of $c_k$ (i.e., $\delta_k$) is bounded by a constant at each iteration $k$.
Intuitively, by keeping the step size below a certain threshold ($\mu_{\min}/2L_{\bmxi}\alpha \sqrt{6d}$), we ensure that the past iterates are not too far from the current iterate.
This reduces the difference between the distribution at past iterates and that at the current iterate, which in turn reduces the estimation error of $c_k$.

From Lemma \ref{lem:delta_bound}, we show the convergence of Algorithm~\ref{alg:zeroth_order_method}.
\begin{theorem} \label{thm:convergence}
Suppose that Assumptions \ref{asm:F_sample_var_bound}--\ref{asm:F_smooth} hold.
Let $\{\bx_k\}$ be the sequence generated by Algorithm~\ref{alg:zeroth_order_method} with $m_k =\Theta(\epsilon^{-2} d^{2})$ for all $k \in \{0, \dots, T\}$, $M:= \frac{L_{\bmxi}^2\alpha^2}{\sigma^2}$, $\mu_{\min}= \Theta(\epsilon d^{-\frac{3}{2}})$, $\mu_0= \Theta(\epsilon d^{-\frac{3}{2}})$ such that $\mu_0 \ge \mu_{\min}$, and $\beta:= \min \left(\frac{1}{12(d+4)H_F}, \frac{1}{(T+1)^{\frac{1}{2}}d^{\frac{3}{4}}}, \frac{\mu_{\min}}{2L_{\bmxi}\alpha \sqrt{6d}} \right)$.
Let $\hat{\bx}:=\bx_{k'}$, where $k'$ is chosen from a uniform distribution over $\{0, \dots, T \}$.
Then, for any $\bx_0 \in \R^d$, $c_0 \in \R$, $s_{\max} \in \{ 1,\dots,T\}$, and $\gamma \in (0,1)$, the iteration complexity required to obtain $\E[\|\nabla F(\hat{\bx})\|^2] \le \epsilon^2$ is $O(d^{\frac{5}{2}} \epsilon^{-4})$.
Moreover, the sample complexity is $O(d^{\frac{9}{2}} \epsilon^{-6})$.
\end{theorem}

\paragraph{Remark 2.}
The iteration complexity of our method is $O(d^{\frac{1}{2}})$ larger than that of the Gaussian homotopy method \citep[Theorem C.2]{iwakiri2022single}, which considers the case where the oracle of the objective function contains noise, but the random variables are independent of the decision variables.
This increase in iteration complexity is reasonable, given that we are dealing with a complex situation where the random variables depend on the decision variables.
Moreover, since the sample complexity of the existing zeroth-order method \citep{liu2023time} for non-convex problem (P) is $O(G^6 d^2 \epsilon^{-6})$,\footnote{\citep{liu2023time} claim that the sample complexity of their method is $O(\epsilon^{-3})$, but in our definition of the stationary point ($\E[\|\nabla F(\hat{\bx})\|^2] \le \epsilon^2$), it becomes $O(\epsilon^{-6})$.} where $G:=\sup_{\bx,\bmxi}|f(\bx,\bmxi)|$, the proposed method has an advantage when $G$ is large or unbounded.

\paragraph{Remark 3.}
In order to set the parameters of Theorem \ref{thm:convergence}, it is necessary to know in advance $\sigma$ of Assumption~\ref{asm:F_sample_var_bound} and $\alpha$ of Assumption \ref{asp:distribution_dist_bound}.
While the existing study \citep{ray2022decision} also needs such information (e.g., $\gamma$ in their paper), it may not be known practically.
In such cases, it is possible to start with a sufficiently large value and estimate it from the information obtained during the iterations.

\section{Proposed Method with Two-point Gradient Estimator}
\subsection{Two-point Gradient Estimator}
We consider the following gradient estimator:
\small
\begin{align}
\bg_2(\bx,\mu,\bu,\bmxi^1,\bmxi^2):= \frac{f(\bx+\mu \bu,\bmxi^1)-f(\bx-\mu \bu,\bmxi^2)}{2\mu}  \bu, \label{eq:two-point_grad_def}
\end{align}
\normalsize
where $\bu \sim \mathcal N(0, I_d)$, $\bmxi^1 \sim D(\bx+\mu \bu)$, $\bmxi^2 \sim D(\bx-\mu \bu)$, and $\mu \in \R_{>0}$.\footnote{Although the two-point estimator is also discussed in \citep{liu2023time} as $\bg_{2\textrm{pt-II}}$, they do not prove the convergence of zeroth-order methods using the estimator without $G=\sup_{\bx,\bmxi}|f(\bx,\bmxi)|$.}

Then, the following lemmas hold for the two-point gradient estimator.
\begin{lemma} \label{lem:F_mu_unbiased_gradient_2}
For any $\bx \in \R^d$ and $\mu \in \R_{>0}$, the following holds.
\begin{align*}
&\E_{u\sim \mathcal N(0, {\mathrm I}_d)} \left[ \E_{\bmxi^1 \sim D(\bx+\mu \bu),\bmxi^2 \sim D(\bx-\mu \bu)}\left[\bg_2(\bx,\mu,\bu,\bmxi^1,\bmxi^2) \right]\right] = \nabla F_{\mu}(\bx). 
\end{align*}
\end{lemma}
\begin{lemma} \label{lem:g_bound_two_point}
Suppose that Assumptions \ref{asm:F_sample_var_bound} and \ref{asm:F_smooth} hold.
For any $\bx \in \R^d$,
the following holds.
\begin{align*}
&\E_{\bu \sim \mathcal N(0, {\mathrm I}_d)}\left[  \E_{ \{\bmxi^{1,j}\}_{j=1}^{m},\{\bmxi^{2,j}\}_{j=1}^{m}} \left[ \left\| \frac{1}{m}\sum_{j=1}^{m} \bg_2(\bx,\mu, \bu,\bmxi^{1,j},\bmxi^{2,j}) \right\|^2  \right]\right] \\
&\le  \frac{3}{2}\mu^2H_F^2 (d+6)^3 +6(d+4)\|\nabla F(\bx)\|^2 + \frac{3\sigma^2d}{2\mu^2 m},
\end{align*}
where $\{\bmxi^{1,j}\}_{j=1}^{m} \sim D(\bx+\mu \bu)$ and $\{\bmxi^{2,j}\}_{j=1}^{m} \sim D(\bx-\mu \bu)$.
\end{lemma}

\subsection{Proposed Method and Theoretical Results}
We propose Algorithm \ref{alg:zeroth_order_method_2}. 
It changes the gradient estimator of Algorithm \ref{alg:zeroth_order_method} to the two-point gradient estimator \eqref{eq:two-point_grad_def}.
From Lemmas \ref{lem:F_mu_unbiased_gradient_2} and \ref{lem:g_bound_two_point}, we show the convergence of Algorithm~\ref{alg:zeroth_order_method_2}.
\begin{theorem} \label{thm:convergence_two-point}
Suppose that Assumptions \ref{asm:F_sample_var_bound}--\ref{asm:F_smooth} hold.
Let $\{\bx_k\}$ be the sequence generated by Algorithm~\ref{alg:zeroth_order_method_2} with $m_k =\Theta(\epsilon^{-2} d^{2})$ for all $k \in \{ 0, \dots, T\}$, $\mu_{\min}= \Theta(\epsilon d^{-\frac{3}{2}})$, $\mu_0= \Theta(\epsilon d^{-\frac{3}{2}})$ such that $\mu_0 \ge \mu_{\min}$, and $\beta:= \min \left(\frac{1}{12(d+4)H_F}, \frac{1}{(T+1)^{\frac{1}{2}}d^{\frac{3}{4}}} \right)$.
Let $\hat{\bx}:=\bx_{k'}$, where $k'$ is chosen from a uniform distribution over $\{0, \dots, T \}$.
Then, for any $\bx_0 \in \R^d$ and $\gamma \in (0,1)$, the iteration complexity required to obtain $\E[\|\nabla F(\hat{\bx})\|^2] \le \epsilon^2$ is $O(d^{\frac{5}{2}} \epsilon^{-4})$.
Moreover, the sample complexity is $O(d^{\frac{9}{2}} \epsilon^{-6})$.
\end{theorem}

\begin{algorithm}[tb]
  \caption{Zeroth-order method with the two-point gradient estimator}
  \label{alg:zeroth_order_method_2}
  \begin{algorithmic}[1]
    \Require $\bx_0$, $\mu_0$, $\mu_{\min}$, $\beta$, $\gamma$, $\{m_k\}_{k=0}^T$, $T$. 
    \For{$k =0, 1,\dots, T$}
    \State Sample $\bu_k$ from $\mathcal N(0, {\mathrm I}_d)$, $\{\bmxi_k^{1,j}\}_{j=1}^{m_k}$ from $D(\bx_k+\mu_k \bu_k)$, and $\{\bmxi_k^{2,j}\}_{j=1}^{m_k}$ from $D(\bx_k-\mu_k \bu_k)$. \label{line:sampling_2}
    \State \footnotesize $\bg_k \gets  \frac{1}{m_k}\sum_{j=1}^{m_k}\bg_2(\bx_k,\mu_k,\bu_k,\bmxi_k^{1,j},\bmxi_k^{2,j})$ %\\
    %where $\bg_2$ is defined by \eqref{eq:two-point_grad_def}.
    \normalsize \label{line:g_k_2}
    \State $\bx_{k+1}\gets \bx_k - \beta \bg_k$ \label{line:x_k_update_2}
    \State $\mu_{k+1}\gets \max(\gamma \mu_k, \mu_{\min})$ \label{line:mu_k_update_2}
    \EndFor
  \end{algorithmic}
\end{algorithm}

\paragraph{Remark 4.}
Algorithm \ref{alg:zeroth_order_method_2} has the same sample complexity as Algorithm \ref{alg:zeroth_order_method}.
However, in practice, if we can use both the one-point and two-point gradient estimators, the two-point gradient estimator often achieves better performance (See the experimental results in Section \ref{subsec:exre}).
On the other hand, there are problems where only a one-point gradient estimator can be used. 
For example, in the case where the distribution changes with time, sampling in the same environment is not possible \citep{ray2022decision,liu2023time}.

\section{Experiments} \label{sec:experiment}
We conducted experiments on an application of \emph{multiproduct pricing} to show that Algorithms \ref{alg:zeroth_order_method} and \ref{alg:zeroth_order_method_2} output solutions with lower objective values compared with the existing methods. 
We performed simulation experiments with real retail data from a supermarket service provider in Japan.\footnote{We used publicly available data, ``New Product Sales Ranking'', provided by KSP-SP Co., Ltd, http://www.ksp-sp.com. Accessed August 15, 2024.} 
The details of our experiments are in the supplementary material.

\subsection{Problem Setup}
We consider a variant of \citep{gallego2014multiproduct} where a seller determines the prices of multiple ($n=10$) products for multiple ($m=40$) buyers.
Each buyer can purchase at most one copy of any product.
Let $\bx :=(x_1,x_2,\dots,x_n) \in \R^n$ be the price vector for the products.
Let $\bmxi \in \{ 0,1,\dots,m\}^{n+1}$ denote a random vector, where $\xi_0$ represents the number of buyers not purchasing any product and $\xi_i$ for $i=1,\dots,n$ represents the number of sales of each product.
Let $s(\bx,\bmxi)$ and $c(\bmxi)$ be real-valued functions representing the sales and costs of products, respectively, and defined as:
\begin{align*}
s(\bx,\bmxi):=\sum_{i=1}^n x_i \xi_i, \quad c(\bmxi):= \sum_{i=1}^n c_i(\xi_i), %\label{eq:sell_cost}
\end{align*}
where
\begin{align*} 
c_i(\xi_i):=
\begin{cases}
2 w_i \xi_i, &\xi_i \le l_i, \\
w_i(\xi_i -l_i) + 2w_i l_i, &l_i < \xi_i \le u_i, \\
3w_i (\xi_i -u_i) + w_i(u_i -l_i) + 2w_i l_i, &\xi_i > u_i.
\end{cases} \nonumber %\label{eq:cost_example_2}
\end{align*}
Here, $l_i:=\frac{0.5m}{n}$, $u_i:=\frac{1.5m}{n}$, and $w_i:= \rho_i \theta_i$,
where $\rho_i$ is the random variable generated from a uniform distribution of
$[0.25, 0.5]$ and $\theta_i$ is the normalized recorded average selling price for each product $i$.
The function $c_i$ represents the case where the cost rate varies with the number of sold products.\footnote{Such piecewise linear cost functions are considered in existing studies \citep{shaw1998algorithm,tunc2016stochastic,ou2017improved}.} 
Then, the revenue-maximizing problem is as follows:
\begin{align*}
    \min_{\bx\in \R^n} \E_{\bmxi \sim D(\bx)} \left[ -s(\bx,\bmxi) + c(\bmxi) \right],
\end{align*}
where $D(\bx)$ is the probability distribution that $\bmxi$ follows.

\paragraph{Settings of unknown parameters for $D(\bx)$.}
We assume that buyers choose one product stochastically.
Each buyer chooses product $i \in I:= \{1,\dots,n\}$ with probability
$p_i(\bx)= \frac{e^{\gamma_i(\theta_i-x_i)}}{a_0+\sum_{j=1}^n e^{\gamma_j(\theta_j-x_j)}}$ 
or does not choose any product with probability
$p_0(\bx)= \frac{a_0}{a_0+\sum_{j=1}^n e^{\gamma_j(\theta_j-x_j)}}$.
Here, we let $\gamma_i:=\frac{2\pi}{\sqrt{6}\alpha_i}$ and let $a_0:=0.1n$.
Then,  $\Pr(\xi_i \mid \bx)=\bigl(
\begin{smallmatrix}
   m \\
   \xi_i
\end{smallmatrix}
\bigl)
p_i(\bx)^{\xi_i}$.
Note that the information of $D(x)$ is not used by each method.

\subsection{Compared Methods and Settings}
We implemented the following methods.
\\
\textbf{Proposed-1 (mini-batch).} We implemented Algorithm~\ref{alg:zeroth_order_method} 
with $\mu_0:=0.19$, $\mu_{\min}:=0.0001$, $c_0:=\sum_{j=1}^{20} f(\bx_0,\bmxi^{j}(\bx_0)), s_{\max}:=10$, $\beta:=0.001\cdot 0.95^{k+1}$, $\gamma =0.95$, $m_k=30+2k$, and $M=0.1$, where $k$ is the current iteration number.
\\
\textbf{Proposed-1 (batch size 1).} We implemented Algorithm~\ref{alg:zeroth_order_method} 
with $m_k=1$ and other same parameters as Proposed-1 (mini-batch).
\\
\textbf{Proposed-2 (mini-batch).} 
We implemented Algorithm~\ref{alg:zeroth_order_method_2} 
with $\mu_0:=0.19$, $\mu_{\min}:=0.0001$, $\beta:=0.001\cdot 0.95^{k+1}$, $m_k=30+2k$, and $\gamma =0.95$.
\\
\textbf{Proposed-2 (batch size 1).} 
We implemented Algorithm~\ref{alg:zeroth_order_method_2} 
with $m_k=1$ and other same parameters as Proposed-2 (mini-batch).
\\
\textbf{CZO-1 (mini-batch).} 
This method is a Conventional Zeroth-Order (CZO) method with a one-point gradient estimator.
It is consistent with Algorithm \ref{alg:zeroth_order_method} where $c_k=0$ for all $k \in \{0,\dots ,T\}$, $\beta=10^{-5}$, $\mu_k=0.001$, and the other parameters are the same as Proposed-1 (mini-batch).
This method is analogous to the zeroth-order method used in existing studies \citep{ray2022decision,liu2023time}.
\\
\textbf{CZO-1 (batch size 1).} 
This method is the same as the CZO-1 (mini-batch) with batch size $m_k=1$ for all $k \in \{0,\dots ,T\}$.

We performed our experiments under the following settings.\\
\textbf{Initial points.}
For all methods, we set the initial points as $\bx_0:=0.5\bme$, where $\bme \in \R^n$ is a vector with all elements $1$.
\\
\textbf{Metric.}
For the output $\hat{\bx}$ of each method, we computed $obj:= \frac{1}{10^3}\sum_{q=1}^{10^3} \left( -s(\hat{\bx},\bmxi^q(\hat{\bx})) + c(\bmxi^q(\hat{\bx}))\right)$, where $\bmxi^q(\hat{\bx}) \sim D(\hat{\bx})$.\\
\textbf{Termination criteria.}
We terminated each method when it had taken 5000 samples from $D(\bx)$ for some $\bx$.

\subsection{Experimental Results} \label{subsec:exre}
Table \ref{tab:real} shows the results of the experiments using real data from different weeks.
All proposed methods were superior to the baselines (CZO-1) for all weeks of data.

Figure \ref{fig2} shows the change in the objective value (\textit{obj}) for each method with respect to the number of samples. 
The figure implies that the proposed methods reduce the objective value more stably than CZO-1. 
In particular, the \hbox{Proposed-2} method shows high performance.
While the Proposed-2 method requires twice as many samples as the Proposed-1 method
at each iteration, it still reduces the objective value with fewer samples. Although the
Proposed-1 method is slightly inferior to the Proposed-2 method, it outperforms the CZO-1 in many cases. 
This fact shows the advantage of introducing
the variance reduction parameter $c$.

\begin{table*}[t]
  \centering
  \caption{
  Results of simulation experiments with real data for 20 randomly generated problem instances. 
The \emph{obj} (\emph{sd}) column represents the average (standard deviation) of the \emph{obj}.
The best value of the average \emph{obj} for each experiment is in bold.
In all experiments, the differences between the proposed method with the best value of the average \emph{obj} and CZO-1 (both types) are significant (two-sided t-test: $p < 0.05$).
  }
  \label{tab:real}
    \scalebox{0.8}{
      \begin{tabular}{crrrrrrrrrrrrrrrrrr} \toprule
date
&\multicolumn{2}{c}{\begin{tabular}{c}
Proposed-1 \\ (mini-batch)\end{tabular}}&
\multicolumn{2}{c}{
\begin{tabular}{c}
Proposed-1 \\ (batch size 1)\end{tabular}}&
\multicolumn{2}{c}{
\begin{tabular}{c}
Proposed-2  \\ (mini-batch)\end{tabular}}&
\multicolumn{2}{c}{
\begin{tabular}{c}
Proposed-2  \\ (batch size 1)\end{tabular}
}&
\multicolumn{2}{c}{
\begin{tabular}{c}
CZO-1 \\ (mini-batch) \end{tabular}
}&
\multicolumn{2}{c}{
\begin{tabular}{c}
CZO-1 \\ (batch size 1)\end{tabular}
}
\\ 
\cmidrule(lr){2-3} \cmidrule(lr){4-5} \cmidrule(lr){6-7} \cmidrule(lr){8-9} \cmidrule(lr){10-11} \cmidrule(lr){12-13}
&\multicolumn{1}{c}{obj}&\multicolumn{1}{c}{sd}& 
\multicolumn{1}{c}{obj}&\multicolumn{1}{c}{sd}&
\multicolumn{1}{c}{obj}&\multicolumn{1}{c}{sd}&
\multicolumn{1}{c}{obj}&\multicolumn{1}{c}{sd}&
\multicolumn{1}{c}{obj}&\multicolumn{1}{c}{sd}&
\multicolumn{1}{c}{obj}&\multicolumn{1}{c}{sd}&
 \\ 
\midrule
02/21--02/27
&\textbf{-7.39}&1.92
&-2.62&4.54
&-7.26&1.68
&-6.62&2.01
&1.83&3.34
&2.83&5.10
\\
03/21--03/27
&-7.61&1.80
&-2.73&4.45
&\textbf{-7.78}&1.48
&-6.70&1.98
&2.41&5.55
&4.08&11.55
\\ 
05/23--05/29
&-5.28&1.74
&-1.89&3.79
&\textbf{-6.38}&1.54
&-5.35&1.77
&0.21&2.88
&0.65&2.94
\\
06/20--06/26
&-5.58&1.69
&1.81&4.62
&-5.54&1.35
&\textbf{-5.61}&1.87
&7.04&9.73
&9.83&11.51
\\
07/18--07/24
&-3.52&2.15
&3.99&3.53
&-4.10&2.02
&\textbf{-4.17}&2.08
&10.39&11.85
&10.29&12.57
\\
08/08--08/14
&-6.40&1.68
&-2.97&4.79
&\textbf{-7.01}&1.29
&-6.17&1.54
&-1.05&1.35
&-0.47&2.05
\\
09/19--09/25
&-3.21	&2.26
&3.32	&4.94
&\textbf{-3.88}	&2.23
&-3.83	&2.44
&10.38	&14.92
&12.15	&12.09
\\
12/05--12/11
&-4.61	&1.78
&2.80	&3.89
&\textbf{-4.88}	&2.10
&-4.27	&2.84
&7.66	&7.58
&5.57	&5.99
\\ \bottomrule
\end{tabular}}
\end{table*}

\begin{figure*}[t!] 
\centering
    \begin{tabular}{c}
        \begin{minipage}{0.25\hsize}
        \centering
        \includegraphics[keepaspectratio, scale=0.25, angle=0]
        {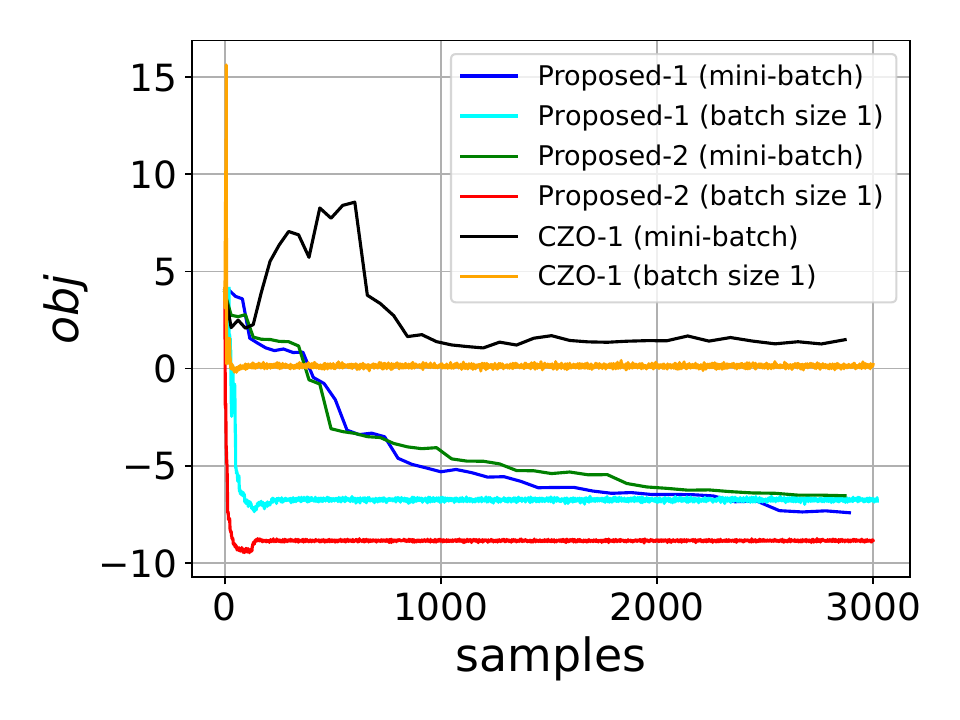}
        \subcaption{02/21--02/27}
        \label{n_ex}
        \end{minipage}
        
        \begin{minipage}{0.25\hsize}
        \centering
        \includegraphics[keepaspectratio, scale=0.25, angle=0]
        {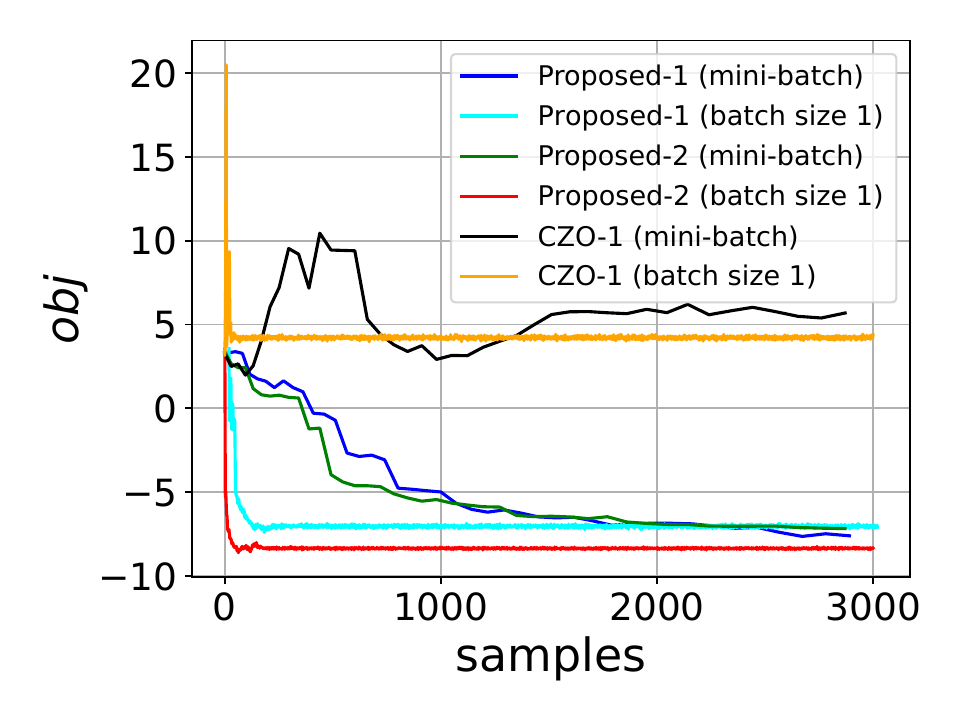}
        \subcaption{03/21--03/27}
        \label{m_ex}
        \end{minipage} 
        
        \begin{minipage}{0.25\hsize}
        \centering
        \includegraphics[keepaspectratio, scale=0.25, angle=0]
        {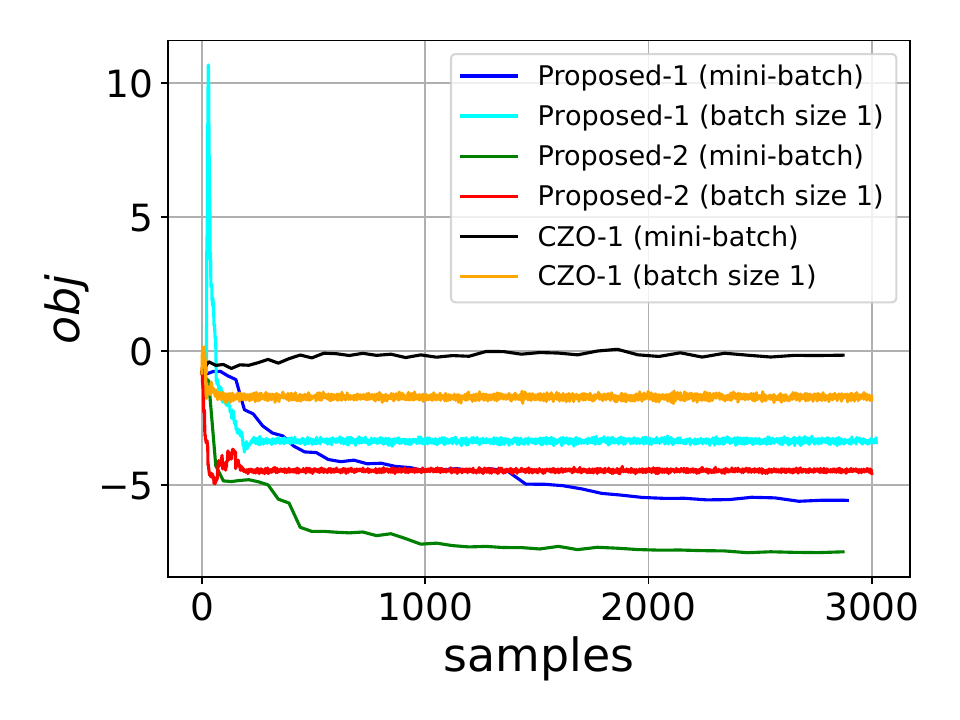}
        \subcaption{05/23--05/29}
        \end{minipage}
        
        \begin{minipage}{0.25\hsize}
        \centering
        \includegraphics[keepaspectratio, scale=0.25, angle=0]
        {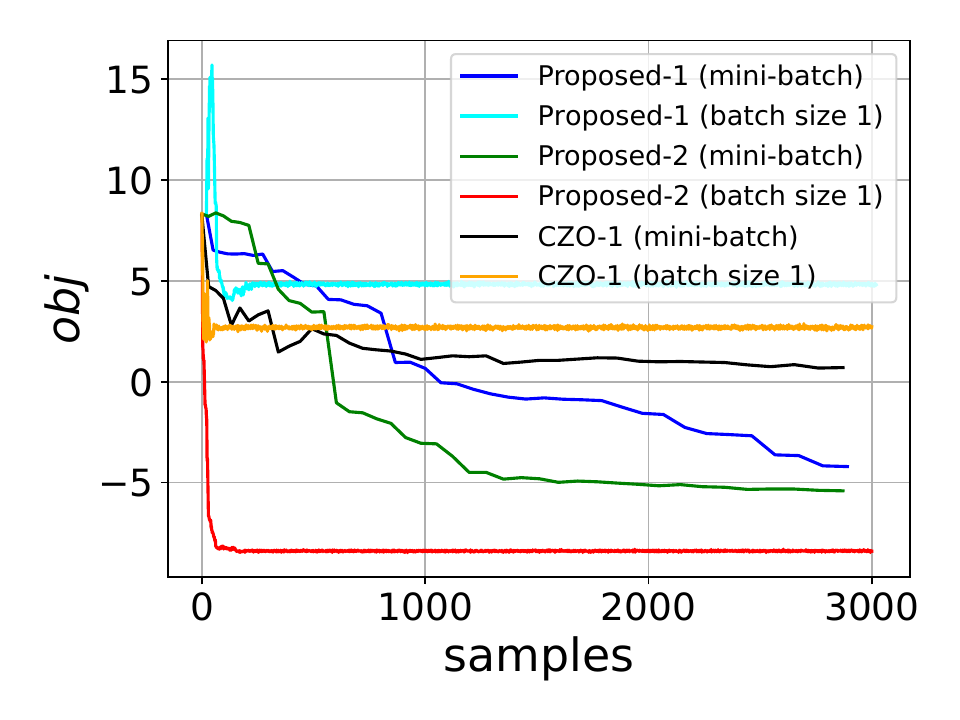}
        \subcaption{06/20--06/26}
        \end{minipage} \vspace{2mm}\\
        
        \begin{minipage}{0.25\hsize}
        \centering
        \includegraphics[keepaspectratio, scale=0.25, angle=0]
        {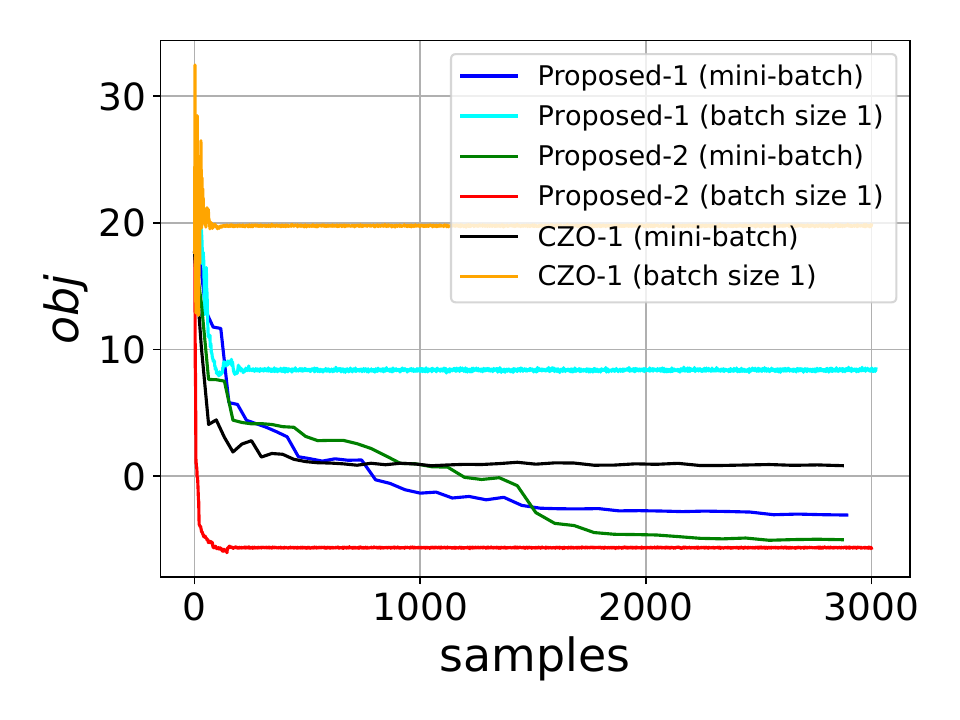}
        \subcaption{07/18--07/24}
        \end{minipage}
        
        \begin{minipage}{0.25\hsize}
        \centering
        \includegraphics[keepaspectratio, scale=0.25, angle=0]
        {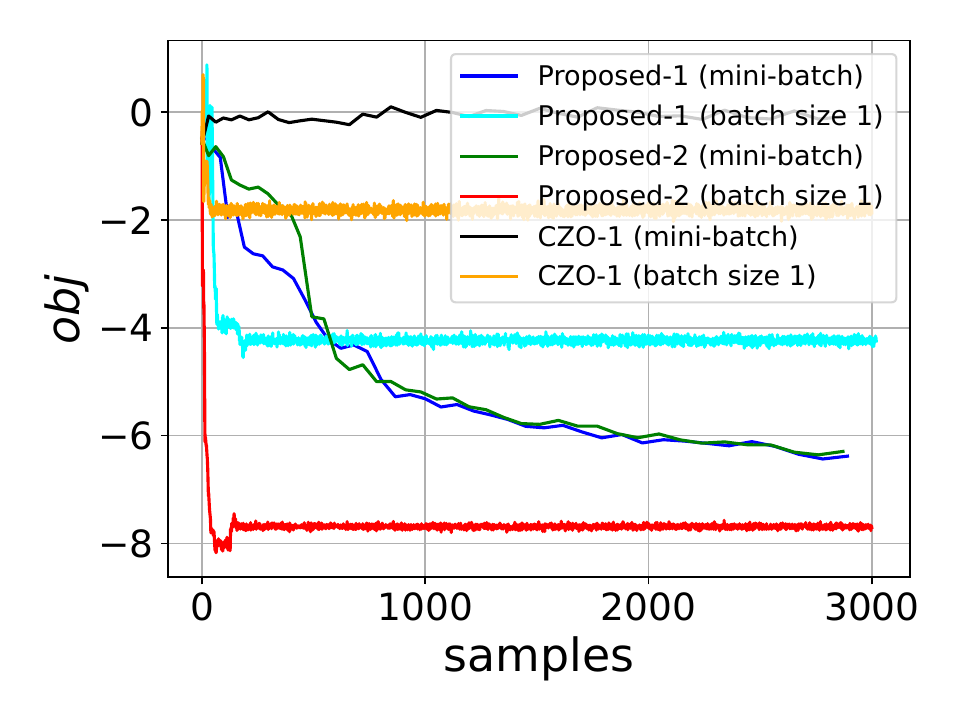}
        \subcaption{08/08--08/14}
        \end{minipage}
        
        \begin{minipage}{0.25\hsize}
        \centering
        \includegraphics[keepaspectratio, scale=0.25, angle=0]
        {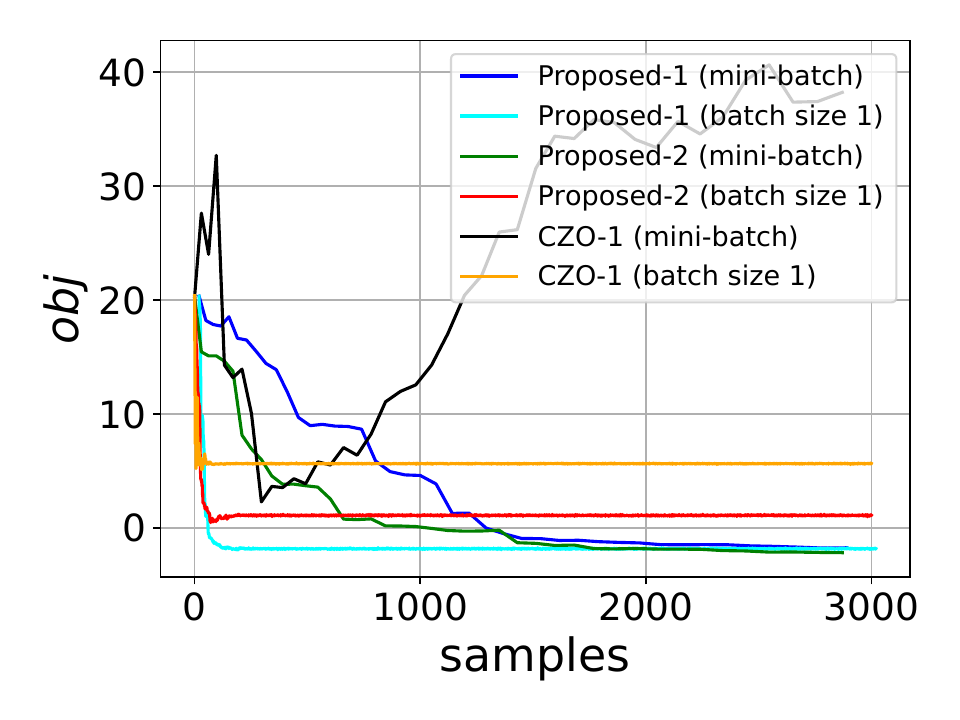}
        \subcaption{09/19--09/25}
        \end{minipage}
        
        \begin{minipage}{0.25\hsize}
        \centering
        \includegraphics[keepaspectratio, scale=0.25, angle=0]
        {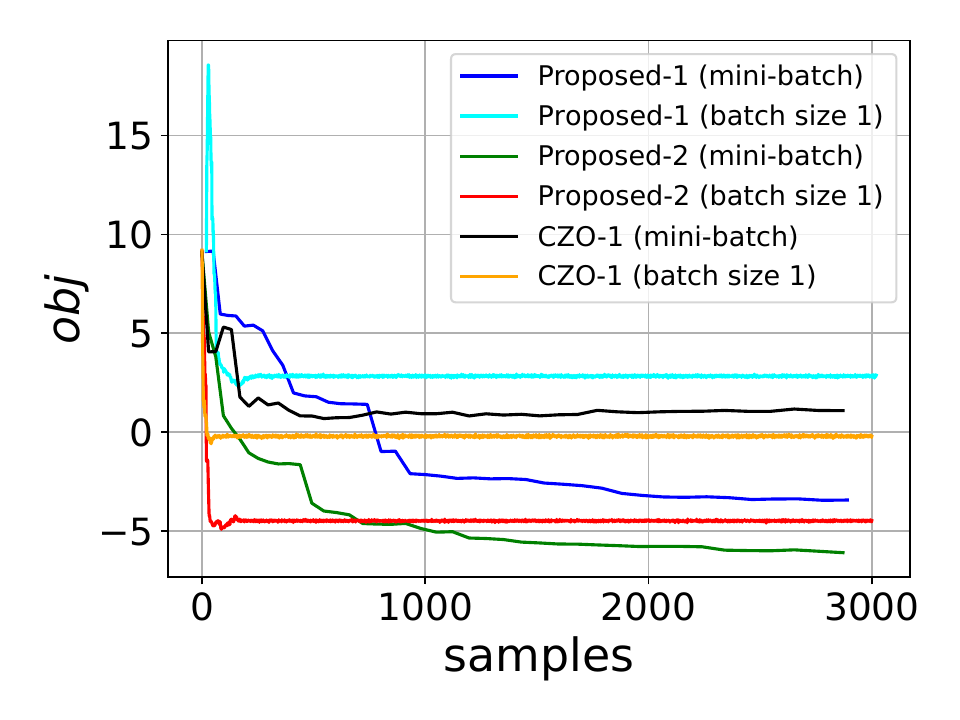}
        \subcaption{12/05--12/11}
        \end{minipage}
        
    \end{tabular}
\caption{
Change in \textit{obj} in the first 3000 samples in the simulation experiment with real data.
Each graph shows the result in one problem instance for each week.
The horizontal axis indicates the number of samples, and the vertical axis indicates \textit{obj}.
}\label{fig2}
\end{figure*}

\section{Conclusion}
We proposed two zeroth-order methods for non-convex problems with decision-dependent uncertainty: 
a method using a new one-point gradient estimator including a variance reduction parameter and the one using a two-point gradient estimator. 
As theoretical results, we showed their convergence  to stationary points and provided the worst-case iteration and sample complexity analysis. 
Our simulation experiments with real data on a retail service application showed good performance of our methods compared to the conventional zeroth-order methods.

In future work, we would like to conduct theoretical analyses of the proposed method in different situations.
For example, one might apply our method to problem (P) where the distribution changes dynamically with time \citep{ray2022decision,liu2023time}.
For such problems, since it is not possible to sample random variables more than once in the same environment, only a one-point gradient estimator can be used.
Therefore, our improved one-point gradient estimator may be effective.
Another future direction is to integrate existing techniques \citep{wang2018stochastic, cai2022zeroth, liu2018zeroth} for zeroth-order methods into our methods.
For example, incorporating the variable selection technique \citep{wang2018stochastic} may enhance the scalability of our method.
Additionally, applying the variance reduction technique \citep{liu2018zeroth} for gradient estimation has the potential to speed up our methods.

\bibliographystyle{abbrvnat}
\bibliography{reference}

\appendix

\section{Details of Our Experiments}
\subsection{Common Settings}
All experiments were conducted on a computer with an AMD EPYC 7413 24-Core Processor, 503.6 GiB of RAM, and Ubuntu 20.04.6 LTS. The program code was implemented in Python 3.8.3.
Our used libraries are the following:
numpy, scipy, math, GPyOpt, time, sys, pickle, matplotlib.

\subsection{Details of Used Data}
We used publicly available data, ``New Product Sales Ranking'', provided by KSP-SP Co., Ltd, http://www.ksp-sp.com.
The data are actual retail data from a middle-size supermarket located in Tokyo.
The parameters of our simulation experiments ($\theta_i$ in our paper) were determined using the sales prices of confectionery products recorded in the data.
We chose to use these data for the following two reasons.
First, since the data are open, we can make the experimental code publicly available.
Second, since the data have been used in an existing price optimization study \citep{ito2016large}, we believe it is appropriate for our experiments on price optimization.

\subsection{Hyper Parameter Settings}
To determine the parameters of the implemented methods, we performed some preliminary experiments under the same setting as the experiments in our paper.
As the parameters of the Proposed-1 (minibatch) method, we adopted the ones that resulted in low objective values.
We tested the following parameters:
$\mu_0\in \{0.019, 0.19, 1.9\}$, $\mu_{\min}\in\{10^{-3},10^{-4},10^{-5}\}$, 
$c_0\in \{\sum_{j=1}^{20} f(\bx_0,\bmxi^{j}(\bx_0)), \sum_{j=1}^{100} f(\bx_0,\bmxi^{j}(\bx_0))\}$, $s_{\max}\in\{1, 10, 100\}$, $\beta_k \in \{10^{-3} \times 0.95^{k+1},10^{-4} \times 0.95^{k+1},10^{-5}\times 0.95^{k+1} \}$, $\gamma \in \{0.9, 0.95, 0.99\}$, $m_k\in \{30, 30+2k, 30+5k\}$, and $M \in\{0.01.0.1,1\}$.
Then, the parameters of the other methods were set based on the parameters of the Proposed-1 method:
for the Proposed-2 (mini-batch) method, all parameters (i.e., $\mu_0$, $\mu_{\min}$, $\beta$, $\gamma$, and $m_k$) were set the same as those of the Proposed-1 method;
for the CZO (mini-batch) method, the parameters $\gamma$ and $m_k$ were set the same as those of the Proposed-1 method, and $\beta:=10^{-5}$ and $\mu_k:=0.001$ for all iteration $k$.
Here, parameters $\beta$ and $\mu_k$ of the CZO method were selected from $\{10^{-3}, 10^{-4},10^{-5},10^{-6},10^{-7}\}$ and $\{10^{-3}, 10^{-4}, 10^{-5}\}$, respectively, with low objective values through preliminary experiments as in the Proposed-1 method.
The Proposed-1 (batch size 1), the Proposed-2 (batch size 1), and the CZO (batch size 1) are identical to the Proposed-1 (mini-batch), the Proposed-2 (mini-batch), and the CZO (mini-batch), respectively, except that $m_k=1$ for all iteration $k$.

\subsection{Random Seed Settings}
We set the random seed in our experiments by ``numpy.random.seed(2024)" in our code.
Note that ``2024" is the used seed.

\section{Proofs}
\subsection{Technical lemmas}
We provide some technical lemmas, which are used to prove the lemmas and theorems in our paper.
\begin{techlemma}
\label{lem:L_F_Lipschitz}
\citep[Lemma 2.1]{jagadeesan2022regret}
Suppose that Assumptions 2 and 3 hold.
Then, $F(\bx)=\E_{\bmxi \sim D(\bx)}[f(\bx,\bmxi)]$ is $L_{F}$-Lipschitz continuous with respect to $\bx$, where $L_{F}:=L_{\bmxi}+\alpha L_{\bx}$.
\end{techlemma}

\begin{techlemma} \label{lem:L_xi_gamma}
\citep[Section 3.1]{jagadeesan2022regret}
Suppose that Assumptions 2 and 3 hold.
Then, for any $\bx \in \R^d$ and $\bx'  \in \R^d$, the following holds:
$$\E_{\bmxi \sim D(\bx')}[f(\bx,\bmxi)] - L_{\bmxi} \alpha \|\bx-\bx'\| \le F(\bx) \le  \E_{\bmxi \sim D(\bx')}[f(\bx,\bmxi)] + L_{\bmxi} \alpha \|\bx-\bx'\|. $$
\end{techlemma}

\begin{techlemma}
\citep[Eq. (14)]{nesterov2017random} \label{lem:nesterov_u_norm}
$$\E_{\bu\sim \mathcal N(0, {\mathrm I}_d)}\left[ \|\bu\|^2  \right]=d.$$
\end{techlemma}

\begin{techlemma} 
\citep[Lemma 3.2]{iwakiri2022single} \label{lem:F_mu_smooth}
When $F$ is $H_F$-smooth, $F_{\mu}$ is $H_F$-smooth for any $\mu \in \R_{>0}$.
\end{techlemma}

\begin{techlemma} 
\citep[Lemma 3.3]{iwakiri2022single} \label{lem:F_mu_lipschitz_wrt_mu}
When $F$ is $L_F$-Lipschitz, $F_{\mu}$ is $L_F\sqrt{d}$-Lipschitz with respect to $\mu$, 
that is, for any $\bx \in \R^d$, $\mu_1 \in \R_{>0}$, and $\mu_2 \in \R_{>0}$,
$$|F_{\mu_1}(\bx) -F_{\mu_2}(\bx) | \le L_F \sqrt{d}|\mu_1-\mu_2|.$$
\end{techlemma}

\begin{techlemma} \label{lem:minibatch_var_reduction}
\citep[p183]{freund1986mathematical}
For any $\bx \in \R^d$, 
$$\E_{\{\bmxi^j\}_{j=1}^m \sim D(\bx)}\left[\left\| \frac{1}{m}\sum_{j=1}^{m} f(\bx, \bmxi^j) - \E_{\bmxi \sim D(\bx)}[f(\bx, \bmxi)] \right\|^2 \right] \le \frac{1}{m} \E_{\bmxi' \sim D}\left[\left\| f(\bx, \bmxi') - \E_{\bmxi \sim D}[f(\bx, \bmxi)] \right\|^2 \right]. $$
\end{techlemma}

\begin{techlemma} \label{lem:nabla_F_nabla_F_mu_relate}
\citep[Lemma 4]{nesterov2017random} 
Suppose that Assumption 4 holds. 
Then, for any $\bx \in \R^d$ and $\mu \in \R_{>0}$,
$$ \|\nabla F(\bx)\|^2 \le 2\|\nabla F_{\mu}(\bx)\|^2 + \frac{\mu^2}{2} H_F^2 (d+6)^3.$$
\end{techlemma}

\begin{techlemma}
\citep[Theorem 4]{nesterov2017random} \label{lem:nesterov_bound}
Suppose that Assumption 4 holds.
Then, for any $\bx\in \R^d$ and $\mu \in \R_{>0}$, the following holds.
$$ \E_{\bu\sim \mathcal N(0, {\mathrm I}_d)}\left[\left\| \frac{1}{\mu} (F(\bx+\mu \bu)-F(\bx)) \bu \right\|^2 \right] \le \frac{\mu^2}{2}H_F^2 (d+6)^3 +2(d+4)\|\nabla F(\bx)\|^2.$$
\end{techlemma}

\paragraph{Remark.}
Lemma \ref{lem:nesterov_bound} corresponds to the first inequality of (35) in \citep[Theorem 4]{nesterov2017random}. 
\cite{nesterov2017random} show the inequality holds under the assumption that $f$ (which is corresponds to $F$ in Lemma \ref{lem:nesterov_bound}) is convex.
However, the inequality holds even if the function $f$ is non-convex. 
This is because the convexity of $f$ is not used in the proof.
Therefore, Lemma \ref{lem:nesterov_bound} does not require the convexity of $F$.

\begin{techlemma}\label{lem:nesterov_bound_2}
Suppose that Assumption 4 holds.
Then, for any $\bx\in \R^d$ and $\mu \in \R_{>0}$, the following holds.
$$ \E_{\bu\sim \mathcal N(0, {\mathrm I}_d)}\left[\left\| \frac{1}{2\mu} (F(\bx+\mu \bu)-F(\bx-\mu \bu))  \bu \right\|^2 \right] \le \frac{\mu^2}{2}H_F^2 (d+6)^3 +2(d+4)\|\nabla F(\bx)\|^2. $$
\end{techlemma}
\begin{proof}
Since $F$ is $H_F$-smooth, 
\begin{align*} 
F(\bx+\mu \bu)-F(\bx-\mu \bu) &=F(\bx+\mu \bu)-F(\bx)+F(\bx)-F(\bx-\mu \bu)\\ 
&\overset{(*)}{\le}\mu\langle\nabla F(\bx),\bu\rangle+\frac{\mu^2}{2}H_F\|\bu\|^2+\mu\langle\nabla F(\bx),\bu\rangle+\frac{\mu^2}{2}H_F \|\bu\|^2\\ &=2\mu\langle\nabla F(\bx),\bu\rangle+\mu^2H_F\|\bu\|^2, \end{align*}
where (*) comes from the fact that $$ \begin{aligned} &F(\bx+\mu \bu)\le F(\bx)+\langle\nabla F(\bx),\mu \bu\rangle+\frac{1}{2}H_F\|\mu \bu\|^2,\\ &F(\bx-\mu \bu)\ge F(\bx)-\langle \nabla F(\bx),\mu \bu\rangle-\frac{1}{2}H_F\| \mu \bu\|^2. \end{aligned} $$
Similarly, we have $F(\bx+\mu \bu)-F(\bx-\mu \bu)\ge2\mu\langle\nabla F(\bx),\bu\rangle-\mu^2H_F\|\bu \|^2.$
Therefore,
\begin{align*}
(F (\bx + \mu \bu) - F (\bx - \mu \bu))^2 \le 8 \mu^2 (\nabla F(\bx)^\top \bu)^2 + 2\mu^4 H_F^2\|\bu\|^4. \end{align*}
Then, 
\begin{align*} 
E_{\bu}\left[\left\|\frac{1}{2\mu}(F(\bx+\mu \bu)-F(\bx-\mu \bu))\bu\right\|^2\right]&=\frac{1}{4\mu^2}E_{\bu}\left[(F(\bx+\mu \bu)-F(\bx-\mu \bu))^2\|\bu\|^2\right]\\ &\le\frac{1}{2\mu^2}\left( E_{\bu}[\mu^4H_F^2\|\bu\|^6]+E_{\bu}[4\mu^2\langle\nabla F(\bx),\bu\rangle^2\|\bu\|^2]\right)\\ 
&\overset{(*)}{\le} \frac{\mu^2}{2}H_F^2(d+6)^3+2(d+4)\|\nabla F(\bx) \|^2, 
\end{align*}
where (*) comes from \citep[(17) and (32)]{nesterov2017random}. 
\end{proof}

\subsection{Proof of Lemma \ref{lem:unbiased_gradient_gaussian_smoothing}}
\setcounter{lemma}{1}
\begin{lemma} \label{lem:F_mu_unbiased_gradient}
For any $\bx \in \R^d, \mu \in \R_{>0}$, and $c\in \R$, the following holds.
$$ \E_{\bu\sim \mathcal N(0, {\mathrm I}_d)} \left[ \E_{\bmxi \sim D(\bx+\mu \bu)}\left[\bg_1(\bx,\mu,c,\bu,\bmxi) \right]\right] = \nabla F_{\mu}(\bx). $$
\end{lemma}
\begin{proof}
\begin{align*}
&\E_{\bu\sim \mathcal N(0, {\mathrm I}_d)} \left[ \E_{\bmxi \sim D(\bx+\mu \bu)}\left[\bg_1(\bx,\mu,c,\bu,\bmxi) \right]\right]\\
&=\E_{\bu \sim \mathcal N(0, {\mathrm I}_d)}\left[  \E_{ \bmxi \sim D(\bx+\mu \bu)} \left[\frac{f(\bx+\mu \bu, \bmxi)-c}{\mu}\bu\right] \right]  \\
&= \E_{\bu \sim \mathcal N(0, {\mathrm I}_d)}\left[ \frac{1}{\mu} \E_{ \bmxi \sim D(\bx+\mu \bu)} [f(\bx+\mu \bu, \bmxi)] \bu\right] - \E_{\bu \sim \mathcal N(0, {\mathrm I}_d)}\left[ \frac{c}{\mu}\bu\right] \\
&\overset{(*)}{=} \E_{\bu\sim \mathcal N(0, {\mathrm I}_d)}\left[\frac{1}{\mu} F(\bx + \mu \bu)\bu\right]  \\
&\overset{(**)}{=} \E_{\bu\sim \mathcal N(0, {\mathrm I}_d)}\left[\frac{1}{\mu} F(\bx + \mu \bu)\bu\right] - \frac{1}{\mu} F(\bx) \E_{\bu\sim \mathcal N(0, {\mathrm I}_d)}\left[ \bu\right]  \\
&= \E_{\bu\sim \mathcal N(0, {\mathrm I}_d)}\left[\frac{1}{\mu} \left( F(\bx + \mu \bu) - F(\bx)\right) \bu\right]  \\
&\overset{(***)}{=} \nabla F_{\mu}(\bx),
\end{align*}
where (*) and (**) follow from the fact that  $\E_{\bu\sim \mathcal N(0, {\mathrm I}_d)}\left[\bu\right] =0$, and (***) comes from \citep[Eq. (21)]{nesterov2017random}.
\end{proof}

\subsection{Proof of Lemma \ref{lem:g_bound}}
\begin{lemma} 
Suppose that Assumptions~1 and 4 hold.
Then, the following holds for any $\bx \in \R^d$, $c \in \R$, and $m \in \mathbb{N}$.
\begin{align*}
&\E_{\bu \sim \mathcal N(0, {\mathrm I}_d)}\left[  \E_{ \{\bmxi^j\}_{j=1}^{m} \sim D(\bx+\mu \bu)} \left[ \left\|\frac{1}{m}\sum_{j=1}^{m} \bg_1(\bx,\mu,c,\bu,\bmxi^j) \right\|^2  \right]\right] \\
&\le  \frac{3}{2}\mu^2H_F^2 (d+6)^3 +6(d+4)\|\nabla F(\bx)\|^2 + \frac{3\sigma^2d}{\mu^2 m} + \frac{3d(F(\bx)-c)^2}{\mu^2}.
\end{align*}
\end{lemma}
\begin{proof}

From the inequality of arithmetic and geometric means, the following holds for any $\bx \in \R^d$, $\by \in \R^d$, and $\bz \in \R^d$:
$$\|\bx\|\|\by\| \leq \frac{\|\bx\|^2 + \|\by\|^2}{2}, \quad \|\by\|\|\bz\| \leq \frac{\|\by\|^2 + \|\bz\|^2}{2}, \quad \|\bz\|\|\bx\| \leq \frac{\|\bz\|^2 + \|\bx\|^2}{2}. $$

Then,
\begin{align}
\|\bx + \by + \bz \|^2 & =  \|\bx\|^2 + \|\by\|^2 + \|\bz\|^2 + 2\bx^\top \by + 2\by^\top \bz + 2\bx^\top \bz \nonumber \\
& \le \|\bx\|^2 + \|\by\|^2 + \|\bz\|^2 + 2\|\bx\| \|\by\| + 2\|\by\|\|\bz\| + 2\|\bx\| \|\bz\| \nonumber \\
& \le 3 \|\bx\|^2 + 3 \|\by\|^2 + 3 \|\bz\|^2. \label{eq:3_vec_bound}  
\end{align}

We then show the inequality of Lemma \ref{lem:g_bound}: 
\footnotesize
\begin{align*}
&\E_{\bu \sim \mathcal N(0, {\mathrm I}_d)}\left[  \E_{ \{\bmxi^j\}_{j=1}^{m} \sim D(\bx+\mu \bu)} \left[ \left\|\frac{1}{m}\sum_{j=1}^{m} \bg_1(\bx,\mu,c,\bu,\bmxi^j) \right\|^2  \right]\right] \\
&= \E_{\bu\sim \mathcal N(0, {\mathrm I}_d)}\left[  \E_{ \{\bmxi^j\}_{j=1}^{m} \sim D(\bx+\mu \bu)} \left[ \left\|\frac{ \frac{1}{m}\sum_{j=1}^{m} f(\bx+\mu \bu, \bmxi^j)-c}{\mu}\bu \right\|^2  \right]\right]  \\
&= \E_{\bu\sim \mathcal N(0, {\mathrm I}_d)}\left[  \E_{\{\bmxi^j\}_{j=1}^{m} \sim D(\bx+\mu \bu)} \left[ \left\| \frac{F(\bx+\mu \bu)-F(\bx)-F(\bx+\mu \bu)+\frac{1}{m}\sum_{j=1}^{m} f(\bx+\mu \bu, \bmxi^j)+F(\bx) -c}{\mu}\bu \right\|^2 \right]\right]  \\
&\overset{(*)}{\le} \E_{\bu\sim \mathcal N(0, {\mathrm I}_d)}\left[  \E_{ \{\bmxi^j\}_{j=1}^{m} \sim D(\bx+\mu \bu)} \left[ 3\left\| \frac{F(\bx+\mu \bu)-F(\bx)}{\mu}\bu \right\|^2 + 3 \left\| \frac{F(\bx+\mu \bu)-\frac{1}{m}\sum_{j=1}^{m} f(\bx+\mu \bu, \bmxi^j)}{\mu}\bu \right\|^2 + 3 \left\| \frac{F(\bx) -c}{\mu}\bu \right\|^2 \right]\right]  \\
&= 3\E_{\bu\sim \mathcal N(0, {\mathrm I}_d)}\left[  \left\|\frac{(F(\bx+\mu \bu)-F(\bx))}{\mu}  \bu \right\|^2\right] + 3\E_{\bu\sim \mathcal N(0, {\mathrm I}_d)}\left[  \E_{\{\bmxi^j\}_{j=1}^{m} \sim D(\bx+\mu \bu)} \left[ \frac{(F(\bx+\mu \bu)-\frac{1}{m}\sum_{j=1}^{m} f(\bx+\mu \bu, \bmxi^j))^2}{\mu^2} \|\bu \|^2 \right]\right] \\
& \quad 
 + \frac{3(F(\bx) -c)^2}{\mu^2} \E_{\bu\sim \mathcal N(0, {\mathrm I}_d)}\left[\|\bu \|^2 \right] \\
    &\overset{(**)}{\le} \frac{3}{2}\mu^2H_F^2 (d+6)^3 +6(d+4)\|\nabla F(\bx)\|^2 + \frac{3\sigma^2d}{\mu^2 m} + \frac{3d(F(\bx) -c)^2}{\mu^2},
\end{align*}
\normalsize
where (*) follows from \eqref{eq:3_vec_bound}, and (**) follows from Assumption~1, Lemmas \ref{lem:nesterov_u_norm}, \ref{lem:minibatch_var_reduction}, and \ref{lem:nesterov_bound}. 
\end{proof}

\subsection{Proof of Lemma \ref{lem:delta_bound_2}}
\begin{lemma}
Let $c_{k} := \sum_{i=k-s}^{k-1} \frac{a_i}{m_i}\sum_{j=1}^{m_i}f(\bx_{k}, \bmxi_i^j)$,
where $\{ \bmxi_i^j \}_{j=1}^{m_i} \sim D(\bx_i + \mu_i \bu_i)$ for each $i \in \{k-s, \dots, k-1\}$.
Then, the following holds:
\begin{align*}
\left( F(\bx_k) - c_k \right)^2 
\le 2 \sum_{i=k-s}^{k-1} a_i^2 \left(L_{\bmxi}^2 \alpha^2 \| \bx_k -\bx_i - \mu_i \bu_i \|^2+  \frac{\sigma^2}{m_i}\right). 
\end{align*}
\end{lemma}
\begin{proof}
\begin{align*}
\delta_k^2 &= \left(\E_{ \bmxi \sim D(\bx_k)} [ f(\bx_k,\bmxi)] - \sum_{i=k-s}^{k-1} \frac{a_i}{m_i}\sum_{j=1}^{m_i}f(\bx_{k}, \bmxi_i^j) \right)^2 \\ 
& =\bigg(\E_{ \bmxi \sim D(\bx_k)} [ f(\bx_k,\bmxi)] 
- \sum_{i=k-s}^{k-1} a_i \E_{ \bmxi \sim D(\bx_i+\mu_i \bu_i)} [ f(\bx_k,\bmxi)] \\
&\quad + \sum_{i=k-s}^{k-1} a_i \E_{ \bmxi \sim D(\bx_i+\mu_i \bu_i)} [ f(\bx_k,\bmxi)] - \sum_{i=k-s}^{k-1} \frac{a_i}{m_i}\sum_{j=1}^{m_i}f(\bx_{k}, \bmxi_i^j) \bigg)^2 \\
&\le 2\left(\E_{ \bmxi \sim D(\bx_k)} [ f(\bx_k,\bmxi)] 
- \sum_{i=k-s}^{k-1} a_i \E_{ \bmxi \sim D(\bx_i+\mu_i \bu_i)} [ f(\bx_k,\bmxi)]  \right)^2 \\
&\quad + 2\left(\sum_{i=k-s}^{k-1} a_i \E_{ \bmxi \sim D(\bx_i+\mu_i \bu_i)} [ f(\bx_k,\bmxi)] 
- \sum_{i=k-s}^{k-1} \frac{a_i}{m_i}\sum_{j=1}^{m_i}f(\bx_{k}, \bmxi_i^j) \right)^2 \\
&\overset{(*)}{=} 2\left(\sum_{i=k-s}^{k-1} a_i \E_{ \bmxi \sim D(\bx_k)} [ f(\bx_k,\bmxi)] 
- \sum_{i=k-s}^{k-1} a_i \E_{ \bmxi \sim D(\bx_i+\mu_i \bu_i)} [ f(\bx_k,\bmxi)]  \right)^2 \\
&\quad + 2\left(\sum_{i=k-s}^{k-1} a_i \E_{ \bmxi \sim D(\bx_i+\mu_i \bu_i)} [ f(\bx_k,\bmxi)] 
- \sum_{i=k-s}^{k-1} \frac{a_i}{m_i}\sum_{j=1}^{m_i}f(\bx_{k}, \bmxi_i^j) \right)^2 \\
&\overset{(**)}{\le} 2 \sum_{i=k-s}^{k-1} \left( a_i \E_{ \bmxi \sim D(\bx_k)} [ f(\bx_k,\bmxi)] 
- a_i \E_{ \bmxi \sim D(\bx_i+\mu_i \bu_i)} [ f(\bx_k,\bmxi)]  \right)^2 \\
&\quad + 2\sum_{i=k-s}^{k-1} \left(a_i \E_{ \bmxi \sim D(\bx_i+\mu_i \bu_i)} [ f(\bx_k,\bmxi)] 
- \frac{a_i}{m_i}\sum_{j=1}^{m_i}f(\bx_{k}, \bmxi_i^j) \right)^2 \\
&= 2 \sum_{i=k-s}^{k-1} a_i^2 \left( \E_{ \bmxi \sim D(\bx_k)} [ f(\bx_k,\bmxi)] 
- \E_{ \bmxi \sim D(\bx_i+\mu_i \bu_i)} [ f(\bx_k,\bmxi)]  \right)^2 \\
&\quad + 2\sum_{i=k-s}^{k-1} a_i^2\left( \E_{ \bmxi \sim D(\bx_i+\mu_i \bu_i)} [ f(\bx_k,\bmxi)] 
- \frac{1}{m_i}\sum_{j=1}^{m_i}f(\bx_{k}, \bmxi_i^j) \right)^2 \\
&\overset{(***)}{\le} 2 \sum_{i=k-s}^{k-1} a_i^2 L_{\bmxi}^2 \alpha^2 \|\bx_k -\bx_i - \mu_i \bu_i \|^2 
+ 2\sum_{i=k-s}^{k-1} a_i^2 \frac{\sigma^2}{m_i} \\
& = 2 \sum_{i=k-s}^{k-1} a_i^2 \left(L_{\bmxi}^2 \alpha^2 \| \bx_k -\bx_i - \mu_i \bu_i \|^2+  \frac{\sigma^2}{m_i}\right), \label{eq:delta_k_previous}
\end{align*}
\normalsize
where (*) follows from $\sum_{i=k-s}^{k-1} a_i=1$, (**) is due to Jensen's inequality and the convexity of L2 norm, (***) comes from Assumption~1, Lemma \ref{lem:L_xi_gamma} and Lemma \ref{lem:minibatch_var_reduction}.

\end{proof}

\subsection{Proof of Lemma \ref{lem:weight_opt_2}}
\begin{lemma}
Consider the following optimization problem:
\begin{align*}
\min_{\ba} \ \ &\sum_{i=k-s}^{k-1} b_i a_i^2  \\
\textrm{s.t.} \ \  \ &a_i \ge0,\  \forall i \in I:= \{k-s, \dots, k-1\}, \\
& \sum_{i=k-s}^{k-1} a_i=1,
\end{align*}
where $b_i:= L_{\bmxi}^2 \alpha^2 \| \bx_k -\bx_i - \mu_i \bu_i \|^2+  \frac{\sigma^2}{m_i}$.
Then, the optimal solution is $\ba^*$ such that $a^*_i:= \frac{1}{b_i \sum_{j=k-s}^{k-1} \frac{1}{b_j}}$.
\end{lemma}

\begin{proof}
Let the Lagrangian function be as follows:
\[
\mathcal{L}(\ba, \bm{\eta}, \lambda) = \sum_{i \in I} b_i a_i^2 + \sum_{i \in I} \eta_i a_i + \lambda \left( \sum_{i \in I} a_i - 1 \right).
\]
The given optimization problem is a convex problem with a strongly convex objective function and satisfies Slater's condition.
Therefore, if $(\ba, \lambda, \bm{\eta})$ satisfies the following KKT conditions, $\ba$ is the optimal solution for the given optimization problem.
\begin{align*}
&\frac{\partial \mathcal{L}(\ba, \bm{\eta}, \lambda)}{\partial a_i}=2b_i a_i + \eta_i + \lambda = 0, \quad \forall i \in I, \\
&\sum_{i \in I} a_i - 1 =0, \\
&\eta_i a_i = 0,\ a_i \ge 0, \eta_i \le 0, \quad \forall i \in I.
\end{align*}
Here, $\ba^*$ such that $a^*_i:= \frac{1}{b_i \sum_{j \in I} \frac{1}{b_j}}$, $\lambda^*:=-\frac{2}{\sum_{j \in I} \frac{1}{b_j}}$, and $\eta_i^*=0$ for all $i \in I$ satisfy the above condition.
\end{proof}

\subsection{Proof of Lemma \ref{lem:delta_bound}}

\begin{lemma} 
Suppose that Assumptions \ref{asm:F_sample_var_bound}--\ref{asm:F_smooth} hold.
Let $\{\bx_k\}$ and $\{c_k\}$ be the sequence generated by Algorithm~1 with $m_k \ge m_{\min}$ for $k \in \{0, 1,\dots, T\}$, $\beta \le \frac{\mu_{\min}}{2L_{\bmxi}\alpha \sqrt{6d}}$, and $M:=\frac{L_{\bmxi}^2\alpha^2}{\sigma^2}$.
Then, for any setting parameter
$\bx_0 \in \R^d$, $\mu_0 \in \R_{>0}$, $\mu_{\min} \le \mu_0$, $c_0 \in \R$, $s_{\max} \in \mathbb{N}$, and $\gamma \in (0,1)$,
the following holds:
\begin{align*}
\E_{\zeta_{[0,k-1]}} [\delta_k^2] \le \left(\frac{1}{2}\right)^{k} \delta_0^2  +  \frac{\mu_{\min}^2\mu_0^2H_F^2(d+6)^3}{2d} + 2\mu_{\min}^2 \frac{d+4}{d} L_F^2 + 8L_{\bmxi}^2 \alpha^2 \mu_0^2 d + \frac{5}{m_{\min}}\sigma^2,
\end{align*}
where $\delta_k:=F(\bx_k)-c_k$, $k \in \{1,\dots, T\}$, and $\zeta_{[0,k-1]}:=\{\bu_{0}, \bmxi_{0}, \dots, \bu_{k-1}, \bmxi_{k-1} \}$.
\normalsize
\end{lemma}
\begin{proof}

From $M=\frac{L_{\bmxi}^2\alpha^2}{\sigma^2}$ and the definition of $a_i$ in line 7 of Algorithm 1, we have
$a_i= \frac{1}{\hat{b}_i \sum_{j=k-s}^{k-1} \frac{1}{\hat{b}_j}}$, where $\hat{b}_i:= L_{\bmxi}^2 \alpha^2 \| \bx_k -\bx_i - \mu_i \bu_i \|^2+  \frac{\sigma^2}{m_i}$.
Then, from Lemma \ref{lem:weight_opt_2}, the following holds for any $\{\theta_i\}_{i=k-s}^{k-1}$ such that $\sum_{i=k-s}^{k-1} \theta_i=1$ and $\theta_i \ge 0$.
\begin{align*}
\sum_{i=k-s}^{k-1} a_i^2 \left(L_{\bmxi}^2 \alpha^2 \|\bx_k -\bx_i - \mu_i \bu_i \|^2+  \frac{\sigma^2}{m_i}\right)
\le \sum_{i=k-s}^{k-1} \theta_i^2 \left(L_{\bmxi}^2 \alpha^2 \|\bx_k -\bx_i - \mu_i \bu_i \|^2+  \frac{\sigma^2}{m_i}\right).    
\end{align*}
Thus, considering the case $\theta_{k-1} =1$ and $\theta_i =0$ for $i \neq k-1$, we have
\begin{align*}
\sum_{i=k-s}^{k-1} a_i^2 \left(L_{\bmxi}^2 \alpha^2 \|\bx_k -\bx_i - \mu_i \bu_i \|^2+  \frac{\sigma^2}{m_i}\right) \le L_{\bmxi}^2 \alpha^2 \|\bx_k -\bx_{k-1} - \mu_{k-1} \bu_{k-1} \|^2+  \frac{\sigma^2}{m_{k-1}}.   
\end{align*}

Then, from Lemma \ref{lem:delta_bound_2}, it yields that
$$ \delta_k^2 \le 2 \left(L_{\bmxi}^2 \alpha^2 \|\bx_k -\bx_{k-1} - \mu_{k-1} \bu_{k-1} \|^2 + \frac{\sigma^2}{m_{k-1}}\right). $$

Taking the expectation with respect to $\zeta_{[0,k-1]}$ for both sides of the above inequality, the following holds for $k \ge 1$ by letting $\zeta_{[0,-1]}:= \emptyset.$
\small
\begin{align*}
 \E_{\zeta_{[0,k-1]}} [\delta_k^2] &\le \E_{\zeta_{[0,k-1]}} \left[2 \left(L_{\bmxi}^2 \alpha^2 \|\bx_k -\bx_{k-1} - \mu_{k-1} \bu_{k-1} \|^2 + \frac{\sigma^2}{m_{k-1}}\right) \right] \\
 & \le \E_{\zeta_{[0,k-1]}} \left[2 \left(L_{\bmxi}^2 \alpha^2 (2\|\bx_k -\bx_{k-1} \|^2 + 2\| \mu_{k-1} \bu_{k-1} \|^2) + \frac{\sigma^2}{m_{k-1}}\right)\right]  \\
 &=  4L_{\bmxi}^2 \alpha^2 \beta^2 \E_{\zeta_{[0,k-1]}} 
 [\| \bg_{k-1}\|^2] + 4L_{\bmxi}^2 \alpha^2 \mu_{k-1}^2 \E_{\zeta_{[0,k-1]}} [\|\bu_{k-1} \|^2] + \frac{2\sigma^2}{m_{k-1}}  \\
& \overset{(*)}{\le} 4L_{\bmxi}^2 \alpha^2 \beta^2 \E_{\zeta_{[0,k-1]}} 
 [\| \bg_{k-1}\|^2] + 4L_{\bmxi}^2 \alpha^2 \mu_{k-1}^2 d + \frac{2\sigma^2}{m_{k-1}} \\
&\overset{(**)}{\le} 4L_{\bmxi}^2 \alpha^2 \beta^2 \left(\frac{3}{2}\mu_{k-1}^2H_F^2 (d+6)^3 +6(d+4) \E_{\zeta_{[0,k-2]}} 
 [\|\nabla F(\bx_{k-1})\|^2] + \frac{3\sigma^2d}{\mu_{k-1}^2 m_{k-1}} + \frac{3d\E_{\zeta_{[0,k-2]}} 
 [\delta_{k-1}^2]}{\mu_{k-1}^2} \right) \\
 & \quad +  4L_{\bmxi}^2 \alpha^2  \mu_{k-1}^2 d + \frac{2\sigma^2}{m_{k-1}}\\
&\overset{(***)}{\le} \frac{12 L_{\bmxi}^2 \alpha^2 \beta^2 d\E_{\zeta_{[0,k-2]}} 
 [\delta_{k-1}^2]}{\mu_{k-1}^2}
 + 6L_{\bmxi}^2 \alpha^2 \beta^2\mu_{k-1}^2H_F^2 (d+6)^3 +24L_{\bmxi}^2 \alpha^2 \beta^2(d+4)L_F^2 
  + \frac{12L_{\bmxi}^2 \alpha^2 \beta^2\sigma^2d}{\mu_{k-1}^2 m_{k-1}}  \\
 &\quad  + 4L_{\bmxi}^2 \alpha^2  \mu_{k-1}^2 d + \frac{2\sigma^2}{m_{k-1}} \\
& \overset{(****)}{\le} \frac{12 L_{\bmxi}^2 \alpha^2 \beta^2 d\E_{\zeta_{[0,k-2]}} 
 [\delta_{k-1}^2]}{\mu_{\min}^2}  + 6L_{\bmxi}^2 \alpha^2 \beta^2\mu_0^2H_F^2 (d+6)^3 +24L_{\bmxi}^2 \alpha^2 \beta^2(d+4)L_F^2 
  + \frac{12L_{\bmxi}^2 \alpha^2 \beta^2\sigma^2d}{\mu_{\min}^2 m_{\min}} \\
 & \quad + 4L_{\bmxi}^2 \alpha^2 \mu_0^2 d + \frac{2\sigma^2}{m_{\min}},
\end{align*}
\normalsize
where (*) follows from Lemma \ref{lem:nesterov_u_norm}, (**) comes from Lemma \ref{lem:g_bound}, (***) is due to Lemma  \ref{lem:L_F_Lipschitz}, and (****) follows from the fact that $\mu_{\min} \le \mu_{k-1} \le \mu_0$ and $m_{\min} \le m_{k-1}$.

Then, since $\beta \le \frac{\mu_{\min}}{2L_{\bmxi}\alpha \sqrt{6d}}$,
\begin{align*}
\E_{\zeta_{[0,k-1]}} 
 [\delta_k^2] & \le\frac{\E_{\zeta_{[0,k-2]}} 
 [\delta_{k-1}^2]}{2}  + \frac{\mu_{\min}^2\mu_0^2H_F^2(d+6)^3}{4d} + \mu_{\min}^2 \frac{d+4}{d} L_F^2  + 4L_{\bmxi}^2 \alpha^2 \mu_0^2 d + \frac{5}{2m_{\min}}\sigma^2\\
 & \le \left(\frac{1}{2}\right)^k \delta_0^2  + \sum_{i=0}^{k-1} \left(\frac{1}{2}\right)^i \left( \frac{\mu_{\min}^2\mu_0^2H_F^2(d+6)^3}{4d} + \mu_{\min}^2 \frac{d+4}{d} L_F^2  + 4L_{\bmxi}^2 \alpha^2 \mu_0^2 d + \frac{5}{2m_{\min}}\sigma^2 \right)  \\
  &= \left(\frac{1}{2}\right)^{k} \delta_0^2  +  \frac{1-\frac{1}{2^{k}}}{1-\frac{1}{2}} \left( \frac{\mu_{\min}^2\mu_0^2H_F^2(d+6)^3}{4d} + \mu_{\min}^2 \frac{d+4}{d} L_F^2  + 4L_{\bmxi}^2 \alpha^2 \mu_0^2 d + \frac{5}{2m_{\min}}\sigma^2 \right) \\
  & \le \left(\frac{1}{2}\right)^{k} \delta_0^2  +  \frac{\mu_{\min}^2\mu_0^2H_F^2(d+6)^3}{2d} + 2\mu_{\min}^2 \frac{d+4}{d} L_F^2  + 8L_{\bmxi}^2 \alpha^2 \mu_0^2 d + \frac{5}{m_{\min}}\sigma^2.
\end{align*}

\end{proof}

\subsection{Proof of Theorem \ref{thm:convergence}}
\setcounter{theorem}{0}
\begin{theorem} 
Suppose that Assumptions \ref{asm:F_sample_var_bound}--\ref{asm:F_smooth} hold.
Let $\{\bx_k\}$ be the sequence generated by Algorithm~1 with $m_k =\Theta(\epsilon^{-2} d^{2})$ for all $k \in \{0, \dots, T\}$, $M:= \frac{L_{\bmxi}^2\alpha^2}{\sigma^2}$, $\mu_{\min}= \Theta(\epsilon d^{-\frac{3}{2}})$, $\mu_0= \Theta(\epsilon d^{-\frac{3}{2}})$ such that $\mu_0 \ge \mu_{\min}$, and $\beta:= \min \left(\frac{1}{12(d+4)H_F}, \frac{1}{(T+1)^{\frac{1}{2}}d^{\frac{3}{4}}}, \frac{\mu_{\min}}{2L_{\bmxi}\alpha \sqrt{6d}} \right)$.
Let $\hat{\bx}:=\bx_{k'}$, where $k'$ is chosen from a uniform distribution over $\{0, \dots, T \}$.
Then, for any $\bx_0 \in \R^d$, $c_0 \in \R$, $s_{\max} \in \{ 1,\dots,T\}$, and $\gamma \in (0,1)$, the iteration complexity required to obtain $\E[\|\nabla F(\hat{\bx})\|^2] \le \epsilon^2$ is $O(d^{\frac{5}{2}} \epsilon^{-4})$.
Moreover, the sample complexity is $O(d^{\frac{9}{2}} \epsilon^{-6})$.
\end{theorem}
\begin{proof}
Let $\Delta_k:= \bg_k - \nabla F_{\mu_k}(\bx_k)$.
From Assumption 4 and Lemma \ref{lem:F_mu_smooth}, $F_{\mu_k}$ is $H_F$-smooth.
Then, 
\begin{align*}
F_{\mu_k}(\bx_{k+1}) &\le F_{\mu_k}(\bx_k) + \nabla F_{\mu_k}(\bx_k)^\top (\bx_{k+1}-\bx_k) + \frac{H_F}{2} \|\bx_{k+1}-\bx_k \|^2 \\
&=F_{\mu_k}(\bx_k) - \beta \nabla F_{\mu_k}(\bx_k)^\top \bg_k + \frac{H_F\beta^2}{2} \|\bg_k \|^2 \\
&= F_{\mu_k}(\bx_k) - \beta \|\nabla F_{\mu_k}(\bx_k)\|^2 -\beta \nabla F_{\mu_k}(\bx_k)^\top \Delta_k + \frac{H_F\beta^2}{2} \|\bg_k \|^2.
\end{align*}
Rearrange the terms in the above
inequality, the following holds.
\begin{align*}
\beta \|\nabla F_{\mu_k}(\bx_k)\|^2 & \le F_{\mu_k}(\bx_k) - F_{\mu_k}(\bx_{k+1}) - \beta \nabla F_{\mu_k}(\bx_k)^\top \Delta_k + \frac{H_F\beta^2}{2} \|\bg_k \|^2 \\
&= F_{\mu_k}(\bx_k) - F_{\mu_{k+1}}(\bx_{k+1}) + F_{\mu_{k+1}}(\bx_{k+1}) - F_{\mu_k}(\bx_{k+1}) + \beta \nabla F_{\mu_k}(\bx)^\top \Delta_k + \frac{H_F\beta^2}{2} \|\bg_k \|^2 \\
& \overset{(*)}{\le} F_{\mu_k}(\bx_k) - F_{\mu_{k+1}}(\bx_{k+1}) + L_F \sqrt{d}|\mu_{k+1}-\mu_k|  +\beta \nabla F_{\mu_k}(\bx)^\top \Delta_k + \frac{H_F\beta^2}{2} \|\bg_k \|^2,
\end{align*}
where (*) holds from Lemma \ref{lem:F_mu_lipschitz_wrt_mu}.
Let $F^*:=\min_{\bx} F(\bx)$.
Summing up the above inequalities for $0 \le k \le T$, we obtain
\begin{align*}
\sum_{k=0}^T \beta \|\nabla F_{\mu_k}(\bx_k)\|^2 & \le F_{\mu_{0}}(\bx_{0}) - F_{\mu_{T+1}}(\bx_{T+1}) + L_F \sqrt{d} \sum_{k=0}^T |\mu_{k+1}-\mu_k| \\
& \quad +\beta \sum_{k=0}^T \nabla F_{\mu_k}(\bx_k)^\top \Delta_k + \frac{H_F\beta^2}{2} \sum_{k=0}^T \|\bg_k \|^2 \\
& \overset{(*)}{\le} F(\bx_{0})+ L_F \sqrt{d} \mu_{0}  - F^* + L_F \sqrt{d} \sum_{k=0}^T |\mu_{k+1}-\mu_k|   \\
& \quad +\beta \sum_{k=0}^T \nabla F_{\mu_k}(\bx_k)^\top \Delta_k + \frac{H_F\beta^2}{2} \sum_{k=0}^T \|\bg_k \|^2 \\
&= F(\bx_{0}) - F^* + L_F \sqrt{d} \left( \mu_{0} + \sum_{k=0}^T |\mu_{k+1}-\mu_k| \right) \\
& \quad +\beta \sum_{k=0}^T \nabla F_{\mu_k}(\bx_k)^\top \Delta_k + \frac{H_F\beta^2}{2} \sum_{k=0}^T \|\bg_k \|^2 \\
&\overset{(**)}{\le } F(\bx_{0}) - F^* + 2 \mu_0 L_F \sqrt{d}  \\
& \quad +\beta \sum_{k=0}^T \nabla F_{\mu_k}(\bx_k)^\top \Delta_k + \frac{H_F\beta^2}{2} \sum_{k=0}^T \|\bg_k \|^2,
\end{align*}
where (*) comes since $F_{\mu_{0}}(\bx_{0}) \le F(\bx_{0})+ L_F \sqrt{d} \mu_{0}$ from Lemma \ref{lem:F_mu_lipschitz_wrt_mu} and $F_{\mu_{T+1}}(\bx_{T+1})=\E_{\bu\sim \mathcal N(0, {\mathrm I}_d)} [F(\bx_{T+1} + \mu_{T+1} \bu)] \ge \E_{\bu\sim \mathcal N(0, {\mathrm I}_d)} [F^*] = F^*$.
The inequality (**) follows from the fact that $\sum_{k=0}^T |\mu_{k+1}-\mu_k| = \sum_{k=0}^T (\mu_k-\mu_{k+1}) = \mu_{0}-\mu_{T+1} \le \mu_0$.
Here, let $\zeta_k:=\{\bu_k, \bmxi_k\}$, $\zeta_{[0,T]}=\{\bu_{0}, \bmxi_{0}, \dots, \bu_T, \bmxi_T \}$, and $m_{\min}:= \min_k m_k$.
Taking the expectation with respect to the random vectors $\zeta_{[0,T]}$, we obtain
\footnotesize
\begin{align}
&\beta \sum_{k=0}^T  \E_{\zeta_{[0,T]}} [\|\nabla F_{\mu_k}(\bx_k)\|^2] \nonumber\\
&\le F(\bx_{0}) - F^* + 2 \mu_0 L_F \sqrt{d} +\beta \sum_{k=0}^T \E_{\zeta_{[0,T]}} [\nabla F_{\mu_k}(\bx_k)^\top \Delta_k] + \frac{H_F\beta^2}{2} \sum_{k=0}^T \E_{\zeta_{[0,T]}}[\|\bg_k \|^2] \nonumber\\
&\overset{(*)}{\le} F(\bx_{0}) - F^* + 2 \mu_0 L_F \sqrt{d} +\beta \sum_{k=0}^T \E_{\zeta_{[0,k-1]}} [ \nabla F_{\mu_k}(\bx_k)^\top \E_{\zeta_k}[\Delta_k \mid \zeta_{[0,k-1]}]] \nonumber\\
& \quad  + \frac{H_F\beta^2}{2} \sum_{k=0}^T \E_{\zeta_{[0,k-1]}} \left[ \frac{3}{2}\mu_k^2H_F^2 (d+6)^3 +6(d+4)\|\nabla F(\bx_k)\|^2 + \frac{3\sigma^2d}{\mu_k^2 m_k} + \frac{3\delta_k^2d}{\mu_k^2} \right] \nonumber\\
&\overset{(**)}{=} F(\bx_{0}) - F^* + 2 \mu_0 L_F \sqrt{d} + \frac{3\beta^2 H_F^3 (d+6)^3}{4} \sum_{k=0}^T \mu_k^2  \nonumber\\
& \quad  
+ 3(d+4)H_F\beta^2 \sum_{k=0}^T \E_{\zeta_{[0,k-1]}} \left[\|\nabla F(\bx_k)\|^2 \right]
+ \frac{3}{2}H_F\beta^2\sigma^2d \sum_{k=0}^T  \frac{1}{\mu_k^2 m_k}
+ \frac{3}{2}H_F\beta^2d \sum_{k=0}^T  
\frac{\E_{\zeta_{[0,k-1]}} [\delta_k^2]}{\mu_k^2} \nonumber \\
& \overset{(***)}{\le} F(\bx_{0}) - F^* + 2 \mu_0 L_F \sqrt{d}  + \frac{3\mu_0^2 \beta^2 H_F^3 (d+6)^3(T+1)}{4}  + 3(d+4)H_F\beta^2 \sum_{k=0}^T \E_{\zeta_{[0,k-1]}} \left[\|\nabla F(\bx_k)\|^2 \right] + \frac{3H_F\beta^2\sigma^2d (T+1)}{2\mu_{\min}^2 m_{\min}}
\nonumber \\
& \quad  
+ \frac{3H_F\beta^2d}{2\mu_{\min}^2} \sum_{k=0}^T  
\E_{\zeta_{[0,k-1]}} [\delta_k^2], \nonumber 
\end{align}
\normalsize
where (*) follows from Lemma \ref{lem:g_bound},
(**) holds from the fact that \\
$\E_{\zeta_k}[\Delta_k \mid \zeta_{[0,k-1]}]=\E_{\zeta_k}[\Delta_k] =\E_{\bu\sim \mathcal N(0, {\mathrm I}_d)}\left[  \E_{ \bmxi \sim D(\bx_k+\mu \bu)} \left[ \Delta_k \right]\right]=0$ from Lemma \ref{lem:F_mu_unbiased_gradient}, and
(***) follows from the fact that $\mu_{\min} \le \mu_k \le \mu_0$ and $m_{\min} \le m_k$ for $k=1,2,\dots,T$.
Then, from Lemma \ref{lem:delta_bound},
\footnotesize
\begin{align}
&\beta \sum_{k=0}^T  \E_{\zeta_{[0,T]}} [\|\nabla F_{\mu_k}(\bx_k)\|^2] \nonumber\\
&\le F(\bx_{0}) - F^* + 2 \mu_0 L_F \sqrt{d} + \frac{3\mu_0^2 \beta^2 H_F^3 (d+6)^3(T+1)}{4} + 3(d+4)H_F\beta^2 \sum_{k=0}^T \E_{\zeta_{[0,k-1]}} \left[\|\nabla F(\bx_k)\|^2 \right] + \frac{3H_F\beta^2\sigma^2d (T+1)}{2\mu_{\min}^2 m_{\min}} \nonumber\\
& \quad + \frac{3H_F\beta^2d}{2\mu_{\min}^2}\sum_{k=0}^T \left(
\left(\frac{1}{2}\right)^{k} \delta_0^2  +  \frac{\mu_{\min}^2\mu_0^2H_F^2(d+6)^3}{2d} + 2\mu_{\min}^2 \frac{d+4}{d} L_F^2  + 8L_{\bmxi}^2 \alpha^2 \mu_0^2 d + \frac{5}{m_{\min}}\sigma^2 \right) \nonumber\\
&=F(\bx_{0}) - F^* + 2 \mu_0 L_F \sqrt{d} + \frac{3\mu_0^2 \beta^2 H_F^3 (d+6)^3(T+1)}{2}  + 3(d+4)H_F\beta^2 \sum_{k=0}^T \E_{\zeta_{[0,k-1]}} \left[\|\nabla F(\bx_k)\|^2 \right] + \frac{9 H_F\beta^2\sigma^2d (T+1)}{\mu_{\min}^2 m_{\min}} \nonumber\\
&
\quad  +\frac{3H_F\beta^2d \delta_0^2}{2\mu_{\min}^2}  \sum_{k=0}^T \frac{1}{2^{k}}  
+ 3 H_F\beta^2 (d+4) L_F^2 (T+1) 
+ \frac{12  H_F\beta^2 L_{\bmxi}^2 \alpha^2 \mu_0^2 d^2(T+1)}{\mu_{\min}^2}  \nonumber\\
&\overset{(*)}{\le} F(\bx_{0}) - F^* + 2 \mu_0 L_F \sqrt{d}  + \frac{3\mu_0^2 \beta^2 H_F^3(d+6)^3(T+1)}{2}  + 3(d+4)H_F\beta^2 \sum_{k=0}^T \E_{\zeta_{[0,k-1]}} \left[\|\nabla F(\bx_k)\|^2 \right] + \frac{9 H_F\beta^2\sigma^2d (T+1)}{\mu_{\min}^2 m_{\min}} \nonumber\\
&\quad 
+ \frac{3H_F\beta^2d \delta_0^2}{\mu_{\min}^2}  
+ 3H_F\beta^2 (d+4)L_F^2(T+1) 
+ \frac{12  H_F\beta^2 L_{\bmxi}^2 \alpha^2 \mu_0^2 d^2(T+1)}{\mu_{\min}^2},
 \label{eq:sum_nabla_F_mu_k_bound}
\end{align}
\normalsize
where (*) holds since $\sum_{k=0}^T \frac{1}{2^{k}}= \frac{1-2^{-(T+1)}}{1-2^{-1}} \le 2$.

Here, from Lemma \ref{lem:nabla_F_nabla_F_mu_relate},
\small
\begin{align*}
\beta \sum_{k=0}^T \E_{\zeta_{[0,T]}} [\|\nabla F(\bx_k)\|^2]
&\le 
2 \beta \sum_{k=0}^T \E_{\zeta_{[0,T]}} [\|\nabla F_{\mu_k}(\bx_k)\|^2] + \frac{\beta H_F^2 (d+6)^3}{2} \sum_{k=0}^T \mu_k^2 \\
& \overset{(*)}{\le} 2\left( F(\bx_{0}) - F^* + 2L_F \mu_0 \sqrt{d} \right) + 3\mu_0^2 \beta^2 H_F^3  (d+6)^3(T+1)  \\
& \quad  + 6(d+4)H_F\beta^2 \sum_{k=0}^T \E_{\zeta_{[0,k-1]}} \left[\|\nabla F(\bx_k)\|^2 \right]
+  \frac{18H_F\beta^2\sigma^2d(T+1)}{\mu_{\min}^2 m_{\min}} \\
& \quad +  \frac{6H_F\beta^2d \delta_0^2}{\mu_{\min}^2} 
+ 6L_F^2H_F\beta^2 (d+4)  (T+1) 
+ \frac{24  H_F\beta^2 L_{\bmxi}^2 \alpha^2 \mu_0^2 d^2 (T+1)}{\mu_{\min}^2}  \\
& \quad  + \frac{\beta H_F^2 (d+6)^3 \mu_0^2(T+1)}{2},
\end{align*}
\normalsize
where (*) comes from  \eqref{eq:sum_nabla_F_mu_k_bound} and the fact that $\mu_k \le \mu_0$ for $k \in \{0,\dots, T\}$.
Rearrange the terms in the above
inequality, we obtain
\begin{align*}
&(\beta-6(d+4)H_F\beta^2) \sum_{k=0}^T \E_{\zeta_{[0,T]}} [\|\nabla F(\bx_k)\|^2] \\
& \le 2\left( F(\bx_{0}) - F^* + 2L_F \mu_0 \sqrt{d} \right) + 3\mu_0^2 \beta^2 H_F^3  (d+6)^3(T+1) 
+  \frac{18H_F\beta^2\sigma^2d(T+1)}{\mu_{\min}^2 m_{\min}} \\
& \quad +  \frac{6H_F\beta^2d \delta_0^2}{\mu_{\min}^2} 
+ 6L_F^2H_F\beta^2 (d+4)  (T+1) 
+ \frac{24  H_F\beta^2 L_{\bmxi}^2 \alpha^2 \mu_0^2 d^2 (T+1)}{\mu_{\min}^2} + \frac{\beta H_F^2 (d+6)^3 \mu_0^2(T+1)}{2}.
\end{align*}

Dividing both sides of the above inequality by $(\beta-6(d+4)H_F\beta^2)(T+1)$ yields
\begin{align*}
&\frac{1}{T+1}\sum_{k=0}^T \E_{\zeta_{[0,T]}} [\|\nabla F(\bx_k)\|^2] \\
& \le 
2\frac{F(\bx_{0}) - F^* +2L_F \mu_0 \sqrt{d}}{(\beta-6(d+4)H_F\beta^2)(T+1)} 
+ \frac{3 \mu_0^2 \beta^2 H_F^3 (d+6)^3}{\beta-6(d+4)H_F\beta^2} 
+ \frac{18 H_F\beta^2 \sigma^2 d}{\mu_{\min}^2m_{\min}(\beta-6(d+4)H_F\beta^2)} \\
& \quad 
+  \frac{6H_F\beta^2d \delta_0^2}{\mu_{\min}^2 (\beta-6(d+4)H_F\beta^2)(T+1)}  + \frac{6L_F^2H_F\beta^2 (d+4)}{\beta-6(d+4)H_F\beta^2}\\
& \quad 
+ \frac{24  H_F\beta^2 L_{\bmxi}^2 \alpha^2  \mu_0^2 d^2}{\mu_{\min}^2(\beta-6(d+4)H_F\beta^2)} 
+ \frac{\beta H_F^2 (d+6)^3 \mu_0^2}{2(\beta-6(d+4)H_F\beta^2)}.
\end{align*}

Moreover, since $\beta= \min \left(\frac{1}{12(d+4)H_F}, \frac{1}{(T+1)^{\frac{1}{2}}d^{\frac{3}{4}}}, \frac{\mu_{\min}}{2L_{\bmxi}\alpha\sqrt{6d}} \right)$, we have
\begin{align*}
  &\frac{1}{\beta-6(d+4)H_F\beta^2} \le \frac{1}{\beta-6(d+4)H_F\beta \cdot \frac{1}{12(d+4)H_F}} = \frac{2}{\beta}, \\  
  &\frac{1}{\beta} \le 12(d+4)H_F + (T+1)^{\frac{1}{2}} d^{\frac{3}{4}} + \frac{2L_{\bmxi}\alpha \sqrt{6d}}{\mu_{\min}}.
\end{align*}

Then,
\begin{align}
&\frac{1}{T+1}\sum_{k=0}^T \E_{\zeta_{[0,T]}} [\|\nabla F(\bx_k)\|^2] \nonumber\\
& \le 4 \frac{F(\bx_{0}) - F^* + 2 \mu_0 L_F \sqrt{d}}{T+1} \left(12(d+4)H_F +  (T+1)^{\frac{1}{2}} d^{\frac{3}{4}}+ \frac{2L_{\bmxi}\alpha \sqrt{6d}}{\mu_{\min}}\right) \nonumber \\
& \quad + 6 \mu_0^2 \beta H_F^3 (d+6)^3
+ \frac{36 H_F\beta \sigma^2 d}{\mu_{\min}^2m_{\min}}  
+  \frac{12 H_F\beta d \delta_0^2}{\mu_{\min}^2(T+1)} 
+ 12L_F^2H_F\beta (d+4)  \nonumber \\ 
& \quad  
+ \frac{48 H_F\beta L_{\bmxi}^2 \alpha^2 \mu_0^2 d^2}{\mu_{\min}^2}
+ H_F^2 (d+6)^3\mu_0^2, \nonumber %\label{eq:mean_nabla_F_bound_progress}
\end{align}
and therefore, we obtain the following.
\begin{align*}
&\frac{1}{T+1}\sum_{k=0}^T \E_{\zeta_{[0,T]}} [\|\nabla F(\bx_k)\|^2] \\
& =  O\left( \left( (T+1)^{-1} + \mu_0 d^{\frac{1}{2}} (T+1)^{-1} \right) \left(d +  (T+1)^{\frac{1}{2}} d^{\frac{3}{4}}  
+  d^{\frac{1}{2}} \mu_{\min}^{-1} \right) \right)  \\
& \quad + O(\mu_0^2 \beta d^3) 
+ O(\beta d \mu_{\min}^{-2} m_{\min}^{-1}) 
+ O(\beta d \mu_{\min}^{-2} (T+1)^{-1}) 
+ O(\beta d) \\
& \quad + O(\beta  \mu_0^2 d^2 \mu_{\min}^{-2}) 
+ O(d^3 \mu_0^2)\\
& \overset{(*)}{=}  O\left( \left( (T+1)^{-1} +  \epsilon d^{-1} (T+1)^{-1} \right) \left(d +  (T+1)^{\frac{1}{2}} d^{\frac{3}{4}}  
+  d^2 \epsilon^{-1} \right) \right)  \\
& \quad + O(d^{-3/4}(T+1)^{-\frac{1}{2}} \epsilon^2) 
+ O(d^{5/4}(T+1)^{-\frac{1}{2}}) 
+  O(d^{13/4}(T+1)^{-\frac{3}{2}}\epsilon^{-2}) 
+ O(d^{1/4}(T+1)^{-\frac{1}{2}}) \\
& \quad + O(d^{5/4}(T+1)^{-\frac{1}{2}}) 
+ O(\epsilon^2),
\end{align*}
where (*) is due to the fact that $\mu_{\min}= \Theta(\epsilon d^{-\frac{3}{2}})$, $\mu_0= \Theta(\epsilon d^{-\frac{3}{2}})$, $m_{\min}=\Theta(\epsilon^{-2}d^{2})$, and $\beta=O(d^{-\frac{3}{4}}(T+1)^{-\frac{1}{2}})$.
Let 
$T:=\Theta(d^{\frac{5}{2}} \epsilon^{-4})$.
Then, 
$$\frac{1}{T+1}\sum_{k=0}^T \E_{\zeta_{[0,T]}} [\|\nabla F(\bx_k)\|^2] \le O(\epsilon^2).$$
Therefore, the iteration complexity is $O(d^{\frac{5}{2}} \epsilon^{-4})$.
Moreover, the sample complexity is $\sum_{k=1}^{T} m_k = T \cdot \Theta(d^2\epsilon^{-2}) = O(d^{\frac{9}{2}}\epsilon^{-6}).$
\end{proof}

\subsection{Proof of Lemma \ref{lem:F_mu_unbiased_gradient_2}}
\setcounter{theorem}{6}
\begin{lemma} 
For any $\bx \in \R^d$ and $\mu \in \R_{>0}$, the following holds.
\begin{align*}
&\E_{\bu\sim \mathcal N(0, {\mathrm I}_d)} \left[ \E_{\bmxi^1 \sim D(\bx+\mu \bu),\bmxi^2 \sim D(\bx-\mu \bu)}\left[\bg_2(\bx,\mu,\bu,\bmxi^1,\bmxi^2) \right]\right] = \nabla F_{\mu}(\bx). 
\end{align*}
\end{lemma}

\begin{proof}
\begin{align*}
&\E_{\bu\sim \mathcal N(0, {\mathrm I}_d)} \left[ \E_{\bmxi^1 \sim D(\bx+\mu \bu),\bmxi^2 \sim D(\bx-\mu \bu)}\left[\bg_2(\bx,\mu,\bu,\bmxi^1,\bmxi^2) \right]\right] \\
&=\E_{\bu \sim \mathcal N(0, {\mathrm I}_d)}\left[  \E_{\bmxi^1 \sim D(\bx+\mu \bu),\bmxi^2 \sim D(\bx-\mu \bu)} \left[\frac{f(\bx+\mu \bu,\bmxi^1)-f(\bx-\mu \bu,\bmxi^2)}{2\mu}  \bu \right] \right]  \\
&= \E_{\bu\sim \mathcal N(0, {\mathrm I}_d)}\left[\frac{F(\bx + \mu \bu) - F(\bx - \mu \bu)}{2\mu} \bu\right]  \\
& \overset{(*)}{=} \nabla F_{\mu}(\bx),
\end{align*}
where (*) comes from \citep[Eq. (26)]{nesterov2017random}.
\end{proof}
\subsection{Proof of Lemma \ref{lem:g_bound_two_point}}
\begin{lemma}
Suppose that Assumptions 1 and 4 hold.
For any $\bx \in \R^d$,
the following holds.
\begin{align*}
&\E_{\bu \sim \mathcal N(0, {\mathrm I}_d)}\left[  \E_{ \{\bmxi^{1,j}\}_{j=1}^{m},\{\bmxi^{2,j}\}_{j=1}^{m}} \left[ \left\| \frac{1}{m}\sum_{j=1}^{m} \bg_2(\bx,\mu, \bu,\bmxi^{1,j},\bmxi^{2,j}) \right\|^2  \right]\right] \\
&\le  \frac{3}{2}\mu^2H_F^2 (d+6)^3 +6(d+4)\|\nabla F(\bx)\|^2 + \frac{3\sigma^2d}{2\mu^2 m},
\end{align*}
where $\{\bmxi^{1,j}\}_{j=1}^{m} \sim D(\bx+\mu \bu)$ and $\{\bmxi^{2,j}\}_{j=1}^{m} \sim D(\bx-\mu \bu)$.
\end{lemma}

\begin{proof}
\small
\begin{align*}
& \E_{\bu\sim \mathcal N(0, {\mathrm I}_d)}\left[  \E_{ \{\bmxi^{1,j}\}_{j=1}^{m},\{\bmxi^{2,j}\}_{j=1}^{m}} \left[ \left\|\frac{1}{m}\sum_{j=1}^{m} \bg_2(\bx,\mu, \bu,\bmxi^{1,j},\bmxi^{2,j})  \right\|^2  \right]\right]  \\
&= \E_{\bu\sim \mathcal N(0, {\mathrm I}_d)}\left[  \E_{ \{\bmxi^{1,j}\}_{j=1}^{m},\{\bmxi^{2,j}\}_{j=1}^{m}} \left[ \left\|\frac{\frac{1}{m}\sum_{j=1}^{m} f(\bx+\mu \bu, \bmxi^{1,j})-\frac{1}{m}\sum_{j=1}^{m} f(\bx-\mu \bu, \bmxi^{2,j}) }{2\mu}\bu \right\|^2  \right]\right]  \\
&= \E_{\bu\sim \mathcal N(0, {\mathrm I}_d)}\bigg[\E_{ \{\bmxi^{1,j}\}_{j=1}^{m},\{\bmxi^{2,j}\}_{j=1}^{m}} \bigg[ \Bigg\| \frac{F(\bx+\mu \bu)-F(\bx - \mu \bu)}{2\mu}\bu + \frac{\frac{1}{m}\sum_{j=1}^{m} f(\bx+\mu \bu, \bmxi^{1,j}) - F(\bx+\mu \bu)}{2\mu}\bu
 \\
& \hspace{50mm} + \frac{-\frac{1}{m}\sum_{j=1}^{m} f(\bx-\mu \bu, \bmxi^{2,j}) +F(\bx-\mu \bu)}{2\mu} \bu\Bigg\|^2 \bigg]\bigg]  \\
&\overset{(*)}{\le} \E_{\bu\sim \mathcal N(0, {\mathrm I}_d)}\Bigg[ 3\left\| \frac{F(\bx+\mu \bu)-F(\bx-\mu \bu)}{2\mu}\bu \right\|^2 
+ 3\E_{ \{\bmxi^{1,j}\}_{j=1}^{m}} \Bigg[ \bigg\| \frac{\frac{1}{m}\sum_{j=1}^{m} f(\bx+\mu \bu, \bmxi^{1,j})-F(\bx+\mu \bu)}{2\mu} \bu \bigg\|^2 \Bigg]
\\
& \hspace{20mm} + 3\E_{ \{\bmxi^{2,j}\}_{j=1}^{m}} \Bigg[ \bigg\| \frac{-\frac{1}{m}\sum_{j=1}^{m} f(\bx-\mu \bu, \bmxi^{2,j}) +F(\bx-\mu \bu)}{2\mu}\bu \bigg\|^2 \Bigg]\Bigg] \\
&\overset{(**)}{\le}  \frac{3}{2}\mu^2H_F^2 (d+6)^3 +6(d+4)\|\nabla F(\bx)\|^2 + \frac{3\sigma^2}{4\mu^2 m}\E_{\bu\sim \mathcal N(0, {\mathrm I}_d)}[\|\bu\|^2] + \frac{3\sigma^2}{4\mu^2 m} \E_{\bu\sim \mathcal N(0, {\mathrm I}_d)}[\|\bu\|^2]  \\
&\overset{(***)}{\le}   \frac{3}{2}\mu^2H_F^2 (d+6)^3 +6(d+4)\|\nabla F(\bx)\|^2 + \frac{3\sigma^2d}{2\mu^2 m},
\end{align*}
\normalsize
where (*) follows from the fact that $\|\bx+\by+\bz\|^2\le 3\|\bx\|^2+3\|\by\|^2+3\|\bz\|^2$ for any $\bx \in \R^d$, $\by \in \R^d$, and $\bz \in \R^d$, as shown in the proof of Lemma \ref{lem:g_bound}, (**) follows from Assumption~1, Lemma \ref{lem:minibatch_var_reduction}, and Lemma~\ref{lem:nesterov_bound_2}.
(***) comes from Lemma \ref{lem:nesterov_u_norm}.
\end{proof}

\subsection{Proof of Theorem \ref{thm:convergence_two-point}}
\setcounter{theorem}{1}
\begin{theorem} 
Suppose that Assumptions \ref{asm:F_sample_var_bound}--\ref{asm:F_smooth} hold.
Let $\{\bx_k\}$ be the sequence generated by Algorithm~\ref{alg:zeroth_order_method_2} with $m_k =\Theta(\epsilon^{-2} d^{2})$ for all $k \in \{ 0, \dots, T\}$, $\mu_{\min}= \Theta(\epsilon d^{-\frac{3}{2}})$, $\mu_0= \Theta(\epsilon d^{-\frac{3}{2}})$ such that $\mu_0 \ge \mu_{\min}$, and $\beta:= \min \left(\frac{1}{12(d+4)H_F}, \frac{1}{(T+1)^{\frac{1}{2}}d^{\frac{3}{4}}} \right)$.
Let $\hat{\bx}:=\bx_{k'}$, where $k'$ is chosen from a uniform distribution over $\{0, \dots, T \}$.
Then, for any $\bx_0 \in \R^d$ and $\gamma \in (0,1)$, the iteration complexity required to obtain $\E[\|\nabla F(\hat{\bx})\|^2] \le \epsilon^2$ is $O(d^{\frac{5}{2}} \epsilon^{-4})$.
Moreover, the sample complexity is $O(d^{\frac{9}{2}} \epsilon^{-6})$.
\end{theorem}
\begin{proof}
Let $\Delta_k:= \bg_k - \nabla F_{\mu_k}(\bx_k)$, $\zeta_k:=\{\bu_k, \bmxi^1_k, \bmxi^2_k\}$, $\zeta_{[0,T]}=\{\bu_{0}, \bmxi^1_{0}, \bmxi^2_{0}, \dots, \bu_T, \bmxi^1_T, \bmxi^2_T \}$, and $m_{\min}:= \min_k m_k$. By the same arguments as in Theorem \ref{thm:convergence}, we obtain the following:
\begin{align*}
&\beta \sum_{k=0}^T  \E_{\zeta_{[0,T]}} [\|\nabla F_{\mu_k}(\bx_k)\|^2] \\
& \le F(\bx_{0}) - F^* +  2 \mu_0 L_F \sqrt{d} +\beta \sum_{k=0}^T \E_{\zeta_{[0,T]}} [\nabla F_{\mu_k}(\bx_k)^\top \Delta_k] + \frac{H_F\beta^2}{2} \sum_{k=0}^T \E_{\zeta_{[0,T]}}[\|\bg_k \|^2].
\end{align*}
Let $[0,-1]:=\emptyset$. 
Then, from Lemma \ref{lem:g_bound_two_point}, we have
\small
\begin{align}
\beta \sum_{k=0}^T  \E_{\zeta_{[0,T]}} [\|\nabla F_{\mu_k}(\bx_k)\|^2]&\le F(\bx_{0}) - F^* +  2 \mu_0 L_F \sqrt{d} +\beta \sum_{k=0}^T \E_{\zeta_{[0,k-1]}} [ \nabla F_{\mu_k}(\bx_k)^\top \E_{\zeta_k}[\Delta_k \mid \zeta_{[0,k-1]}]] \nonumber\\
& \quad  + \frac{H_F\beta^2}{2} \sum_{k=0}^T \E_{\zeta_{[0,k-1]}} \left[ \frac{3}{2}\mu_k^2H_F^2 (d+6)^3 +6(d+4)\|\nabla F(\bx_k)\|^2 + \frac{3\sigma^2d}{2\mu_k^2 m_k} \right] \nonumber\\
&\overset{(*)}{=} F(\bx_{0}) - F^* +  2 \mu_0 L_F \sqrt{d} + \frac{3\beta^2 H_F^3 (d+6)^3}{4} \sum_{k=0}^T \mu_k^2  \nonumber\\
& \quad  
+ 3H_F\beta^2(d+4)\sum_{k=0}^T \E_{\zeta_{[0,k-1]}} \left[\|\nabla F(\bx_k)\|^2 \right]
+ \frac{3}{4}H_F\beta^2\sigma^2d \sum_{k=0}^T  \frac{1}{\mu_k^2 m_k} \nonumber \\
&\overset{(**)}{\le} F(\bx_{0}) - F^* +  2 \mu_0 L_F \sqrt{d} + \frac{3\beta^2 H_F^3 (d+6)^3 (T+1) \mu_0^2}{4} \nonumber\\
& \quad  + 3H_F\beta^2(d+4) \sum_{k=0}^T \E_{\zeta_{[0,k-1]}} \left[\|\nabla F(\bx_k)\|^2 \right]
+  \frac{3H_F\beta^2\sigma^2d(T+1)}{4\mu_{\min}^2 m_{\min}}, 
 \label{eq:sum_nabla_F_mu_k_bound_2}
\end{align}
\normalsize
where (*) holds since $\E_{\zeta_k}[\Delta_k \mid \zeta_{[0,k-1]}]=\E_{\zeta_k}[\Delta_k] =\E_{\bu\sim \mathcal N(0, {\mathrm I}_d)}\left[  \E_{ \bmxi \sim D(\bx_k+\mu \bu)} \left[ \Delta_k \right]\right]=0$ from Lemma~\ref{lem:F_mu_unbiased_gradient_2}.
The inequality (**) holds since $\mu_{\min} \le \mu_k$ and $m_{\min} \le m_k$ for $k=1,2,\dots,T$.
Then, from Lemma \ref{lem:nabla_F_nabla_F_mu_relate},
\small
\begin{align*}
&\beta \sum_{k=0}^T \E_{\zeta_{[0,T]}} [\|\nabla F(\bx_k)\|^2] \\
&\le 2 \beta \sum_{k=0}^T \E_{\zeta_{[0,T]}} [\|\nabla F_{\mu_k}(\bx_k)\|^2] + \frac{\beta H_F^2 (d+6)^3}{2} \sum_{k=0}^T \mu_k^2 \\
& \overset{(*)}{\le} 2\left( F(\bx_{0}) - F^* + 2  \mu_0 L_F \sqrt{d} \right) + \frac{3\beta^2 H_F^3 (d+6)^3 (T+1) \mu_0^2}{2}  \nonumber\\
& \quad  + 6(d+4)H_F\beta^2 \sum_{k=0}^T \E_{\zeta_{[0,k-1]}} \left[\|\nabla F(\bx_k)\|^2 \right]
+  \frac{3H_F\beta^2d\sigma^2(T+1)}{2\mu_{\min}^2 m_{\min}} 
+ \frac{\beta H_F^2 (d+6)^3 (T+1) \mu_0^2}{2} ,
\end{align*}
\normalsize
where (*) comes from \eqref{eq:sum_nabla_F_mu_k_bound_2} and $\mu_k \le \mu_0$.
Rearrange the terms in the above
inequality, we obtain
\small
\begin{align*}
&(\beta-6(d+4)H_F\beta^2) \sum_{k=0}^T \E_{\zeta_{[0,T]}} [\|\nabla F(\bx_k)\|^2] \\
& \le 2\left( F(\bx_{0}) - F^* + 2  \mu_0 L_F \sqrt{d} \right) + \frac{3\beta^2 H_F^3 (d+6)^3 (T+1)\mu_0^2}{2} +  \frac{3H_F\beta^2d\sigma^2(T+1)}{2\mu_{\min}^2 m_{\min}} + \frac{\beta H_F^2 (d+6)^3 (T+1) \mu_0^2}{2}.
\end{align*}
\normalsize

Dividing both sides of the above inequality by $(\beta-6(d+4)H_F\beta^2)(T+1)$ yields
\begin{align*}
\frac{1}{T+1}\sum_{k=0}^T \E_{\zeta_{[0,T]}} [\|\nabla F(\bx_k)\|^2] 
& \le 
2\frac{F(\bx_{0}) - F^* +  2 \mu_0 L_F \sqrt{d}}{(\beta-6(d+4)H_F\beta^2)(T+1)} 
+ \frac{3\beta^2 H_F^3 (d+6)^3 \mu_0^2}{2(\beta-6(d+4)H_F\beta^2)}  \\
& \quad 
+ \frac{3 H_F\beta^2d \sigma^2}{2\mu_{\min}^2m_{\min}(\beta-6(d+4)H_F\beta^2)}  
+ \frac{\beta H_F^2 (d+6)^3 \mu_0^2}{2(\beta-6(d+4)H_F\beta^2)}.
\end{align*}

Here, since $\beta= \min \left(\frac{1}{12(d+4)H_F}, \frac{1}{(T+1)^{\frac{1}{2}}d^{\frac{3}{4}}} \right)$, we have
\begin{align*}
&\frac{1}{\beta-6(d+4)H_F\beta^2} \le \frac{1}{\beta-6(d+4)H_F\beta \cdot \frac{1}{12(d+4)H_F} } = \frac{2}{\beta}, \\
&\frac{1}{\beta} \le 12(d+4)H_F + (T+1)^{\frac{1}{2}} d^{\frac{3}{4}}.
\end{align*}

Then,
\begin{align*}
\frac{1}{T+1}\sum_{k=0}^T \E_{\zeta_{[0,T]}} [\|\nabla F(\bx_k)\|^2] & \le 4 \frac{F(\bx_{0}) - F^* +  2 \mu_0 L_F \sqrt{d}}{T+1} \left(12(d+4)H_F + (T+1)^{\frac{1}{2}} d^{\frac{3}{4}}\right)  \\
& \quad + 3\beta H_F^3 (d+6)^3\mu_0^2
+ \frac{3 H_F\beta d \sigma^2}{\mu_{\min}^2m_{\min}}   
+ H_F^2 (d+6)^3 \mu_0^2. 
\end{align*}

Therefore, we obtain 
\begin{align*}
\frac{1}{T+1}\sum_{k=0}^T \E_{\zeta_{[0,T]}} [\|\nabla F(\bx_k)\|^2] & = O((T+1)^{-1}+\mu_0d^{\frac{1}{2}}(T+1)^{-1})(d +  (T+1)^{\frac{1}{2}} d^{\frac{3}{4}}) \\
& \quad +O(\beta d^3 \mu_0^2) + O(\beta d \mu_{\min}^{-2} m_{\min}^{-1})  + O(d^3 \mu_0^2) \\
& \overset{(*)}{=}   O\left(\left( (T+1)^{-1} + d^{-1} (T+1)^{-1}\epsilon \right) \left(d + (T+1)^{\frac{1}{2}}d^{\frac{3}{4}} \right) \right) \\
& \quad + O(d^{-\frac{3}{4}}(T+1)^{-\frac{1}{2}} \epsilon^2) 
+ O(d^{\frac{5}{4}}(T+1)^{-\frac{1}{2}}) 
+ O(\epsilon^2),\\
& = O\left(d(T+1)^{-1} 
+d^{\frac{3}{4}}(T+1)^{-\frac{1}{2}}
+ (T+1)^{-1}\epsilon 
+ d^{-\frac{1}{4}} (T+1)^{-\frac{1}{2}}\epsilon  \right) \\
& \quad + O(d^{-\frac{3}{4}}(T+1)^{-\frac{1}{2}} \epsilon^2) 
+ O(d^{\frac{5}{4}}(T+1)^{-\frac{1}{2}}) 
+ O(\epsilon^2),
\end{align*}
where (*) follows from the fact that 
$\mu_{\min}= \Theta(\epsilon d^{-\frac{3}{2}})$, $\mu_0= \Theta(\epsilon d^{-\frac{3}{2}})$, $\frac{1}{m_{\min}}=O(\epsilon^2d^{-2})$, and $\beta=O((T+1)^{-\frac{1}{2}}d^{-\frac{3}{4}})$.
Let 
$T:= \Theta(d^{\frac{5}{2}} \epsilon^{-4})$.
Then, 
$$\frac{1}{T+1}\sum_{k=0}^T \E_{\zeta_{[0,T]}} [\|\nabla F(\bx_k)\|^2] \le O(\epsilon^2).$$

Therefore, the iteration complexity is $O(d^{\frac{5}{2}} \epsilon^{-4})$.
Moreover, the sample complexity is $\sum_{k=1}^{T} m_k = T \cdot \Theta(d^2\epsilon^{-2}) = O(d^{\frac{9}{2}}\epsilon^{-6}).$

\end{proof}

\end{document}

%% file: head/package.tex
\usepackage[margin=1.0truein]{geometry}
\usepackage{setspace}
\usepackage{lscape}
\usepackage{pdflscape}

\usepackage{xcolor}

\usepackage{amsmath,amssymb,amsthm} % AMS系，ほぼ必須
\usepackage{wrapfig} % wrapfigure 用

\usepackage{mathrsfs} % 花文字フォント (\mathscr)
\usepackage{url} % url 用フォント (\url)
\usepackage{latexsym} % マニアックな数学記号用
\usepackage{bm} % 斜体太字フォント (\bm)

\usepackage{mathtools}
\usepackage{booktabs,multirow}
\usepackage{array}
\newcolumntype{P}[1]{>{\centering\arraybackslash}p{#1}}
\usepackage{diagbox}
\usepackage{threeparttable}
\usepackage{siunitx} % for scientific notation
\usepackage{algorithm}
\usepackage[noend]{algpseudocode}

\algnewcommand{\Break}{\textbf{break}}
\algnewcommand{\algorithmicgoto}{\textbf{go to} Line}
\algnewcommand{\Goto}[1]{\algorithmicgoto~\ref{#1}}
\algblockdefx{MRepeat}{EndRepeat}{\textbf{repeat}}{}
\algnotext{EndRepeat}
\usepackage{natbib}
\usepackage{thmtools,thm-restate}

\usepackage{hyperref}
\usepackage{cleveref}
\expandafter\def\csname ver@etex.sty\endcsname{3000/12/31} % autonum のバグ対策

\usepackage{multicol}
\usepackage{wasysym} % 顔文字
\usepackage{adjustbox}
\usepackage{xparse} % 複雑なマクロ定義用
\usepackage{pbox} % pbox コマンド用
\usepackage{authblk} % 著者名
\usepackage{xspace} % newcommand マクロ直後に適切にスペースを入れる

\usepackage[subrefformat=parens]{subcaption}
\usepackage{enumitem} % enumerate package より詳細な設定
\hypersetup{
  colorlinks,
  linkcolor={green!35!black},
  citecolor={green!35!black},
  urlcolor={blue!50!black}
}

\usepackage{times}
\usepackage{soul}
\usepackage{secdot}
\usepackage{cases}
\usepackage{subcaption}
\usepackage{caption}

%% file: head/thm.tex
\declaretheoremstyle[
  shaded={bgcolor=gray!15},
]{thmsty}

\declaretheorem[
  name=Theorem,
  refname={Theorem,Theorems},
  style=thmsty,
]{theorem}

\declaretheorem[
  name=Lemma,
  refname={Lemma,Lemmas},
  style=thmsty,
]{lemma}
\declaretheorem[
  name=Definition,
  refname={Definition,Definitions},
  style=thmsty,
]{definition}
\declaretheorem[
  name=Assumption,
  refname={Assumption,Assumptions},
  style=thmsty,
]{assumption}

\declaretheorem[
  name=Lemma,
  refname={Lemma,Lemmas},
  style=thmsty,
]{techlemma}
\usepackage{chngcntr}
\counterwithin{techlemma}{section}

\crefname{algorithm}{Algorithm}{Algorithms}
\crefname{line}{Line}{Lines}
\crefname{section}{Section}{Sections}
\crefname{appendix}{Appendix}{Appendices}
\crefname{table}{Table}{Tables}
\crefname{figure}{Figure}{Figures}
\crefname{equation}{}{}
\Crefname{equation}{Eq.}{Eqs.}

\setlist[itemize]{
  topsep=0.4\baselineskip,
  itemsep=0\baselineskip,
  leftmargin=1.5em,
}

\setlist[enumerate]{
  font=\upshape,
  label=(\alph*),
  ref=(\alph*),
  topsep=0.4\baselineskip,
  itemsep=0\baselineskip,
  leftmargin=2em,
}

\newlist{enuminasm}{enumerate}{1} % asm 環境内で enumerate を使うときの環境
\setlist[enuminasm]{
  font=\upshape,
  label=(\alph*),
  ref=\theassumption(\alph*),
  topsep=0.4\baselineskip,
  itemsep=0\baselineskip,
  leftmargin=2em,
}
\crefalias{enuminasmi}{assumption}

\newlist{enuminthm}{enumerate}{1}
\setlist[enuminthm]{
  font=\upshape,
  label=(\alph*),
  ref=\thetheorem(\alph*),
  topsep=0.4\baselineskip,
  itemsep=0\baselineskip,
  leftmargin=2em,
}
\crefalias{enuminthmi}{theorem}

\newlist{enuminlem}{enumerate}{1}
\setlist[enuminlem]{
  font=\upshape,
  label=(\alph*),
  ref=\thelemma(\alph*),
  topsep=0.4\baselineskip,
  itemsep=0\baselineskip,
  leftmargin=2em,
}
\crefalias{enuminlemi}{lemma}

\newlist{enuminprop}{enumerate}{1}
\setlist[enuminprop]{
  font=\upshape,
  label=(\alph*),
  ref=\theproposition(\alph*),
  topsep=0.4\baselineskip,
  itemsep=0\baselineskip,
  leftmargin=2em,
}
\crefalias{enuminpropi}{proposition}

\newlist{enumincond}{enumerate}{1}
\setlist[enumincond]{
  font=\upshape,
  label=(\alph*),
  ref=\thecondition(\alph*),
  topsep=0.4\baselineskip,
  itemsep=0\baselineskip,
  leftmargin=2em,
}
\crefalias{enumincondi}{condition}

%% file: head/macro.tex
\DeclareMathOperator{\EE}{\mathbb{E}}

\DeclarePairedDelimiterX\inner[2]{\langle}{\rangle}{{#1},{#2}}

\DeclarePairedDelimiter\bra{[}{]}
\DeclarePairedDelimiterX\Set[2]{\{}{\}}{\mspace{2mu}{#1}\;\delimsize|\;{#2}\mspace{2mu}}
\DeclarePairedDelimiterX\Prn[2]{(}{)}{\mspace{2mu}{#1}\;\delimsize|\;{#2}\mspace{2mu}}
\DeclarePairedDelimiterX\Bra[2]{[}{]}{\mspace{2mu}{#1}\;\delimsize|\;{#2}\mspace{2mu}}

\newcommand{\R}{\mathbb R}

\newcommand{\bx}{\bm{x}}

\newcommand{\ba}{\bm{a}}

\newcommand{\bz}{\bm{z}}

\newcommand{\bv}{\bm{v}}

\newcommand{\bu}{\bm{u}}

\newcommand{\bA}{\bm{A}}
\newcommand{\bme}{\bm{e}}
\newcommand{\by}{\bm{y}}
\newcommand{\bmxi}{\bm{\xi}}

\newcommand{\bg}{\bm{g}}

\newcommand{\bnu}{\bm{\nu}}

\newcommand{\mR}{\mathbb{R}}

\newcommand{\E}{\mathbb{E}}
\newcommand{\mpr}{\mathrm{Pr}}

\renewcommand{\epsilon}{\varepsilon}
\NewDocumentCommand{\exsub}{s m O{} m}{%
  \IfBooleanT{#1}{\EE_{#2}\nolimits\bra*{#4}}%
  \IfBooleanF{#1}{\EE_{#2}\nolimits\bra[#3]{#4}}%
}
\NewDocumentCommand{\ex}{s O{} m}{%
  \IfBooleanT{#1}{\EE\nolimits\bra*{#3}}%
  \IfBooleanF{#1}{\EE\nolimits\bra[#2]{#3}}%
}
\NewDocumentCommand{\cex}{s O{} m m}{%
  \IfBooleanT{#1}{\EE\nolimits\Bra*{#3}{#4}}%
  \IfBooleanF{#1}{\EE\nolimits\Bra[#2]{#3}{#4}}%
}
\usepackage{CJKutf8}

\newcommand{\email}[1]{\href{mailto:#1}{\nolinkurl{#1}}}